\documentclass[a4paper,11pt]{article}
\usepackage[latin9]{inputenc}
\usepackage{amsfonts}
\usepackage{amsmath}
\usepackage{amssymb}
\usepackage{amsthm}
\usepackage[english]{babel}
\usepackage{caption}
\usepackage{fontenc}
\usepackage{graphicx}
\usepackage{fancyhdr}
\usepackage{etoolbox}
\usepackage{url}
\usepackage{verbatim}
\usepackage{color}
\usepackage{xcolor}
\usepackage{mdframed}

\title{Revisiting Relations between Stochastic Ageing and Dependence
\\for Exchangeable Lifetimes\\ with
an Extension for the IFRA/DFRA Property}

\newcommand{\lcaporali}{{\fontencoding{T1}\fontfamily{ptm}\selectfont<<}}
\newcommand{\rcaporali}{{\fontencoding{T1}\fontfamily{ptm}\selectfont>>}}
\newtheorem{theorem}{Theorem}[section]

\newtheorem{definition}[theorem]{Definition}
\newtheorem{example}[theorem]{Example}

\newtheorem{lemma}[theorem]{Lemma}

\newtheorem{proposition}[theorem]{Proposition}
\newtheorem{remark}[theorem]{Remark}

\begin{document}

\author{ Giovanna Nappo\textsuperscript{a}\thanks{Sapienza, Universit\`{a} di Roma,  P.le A.~Moro, 5 - I-00185 Roma, Italy; Email: nappo@mat.uniroma1.it}          and
        Fabio Spizzichino\textsuperscript{b}\thanks{Sapienza, Universit\`{a} di Roma,  P.le A.~Moro,5 - I-00185 Roma, Italy;  Email: fabio.spizzichino@fondazione.uniroma1.it}
  }

\date{$ $}

\maketitle

\begin{abstract}
We first review an approach that had been developed in the past years to
introduce concepts of \textquotedblleft bivariate ageing" for  exchangeable
lifetimes and to analyze mutual relations among stochastic dependence,
univariate ageing, and bivariate ageing.

A specific feature of such an approach dwells on the concept of semi-copula
and in the extension, from copulas to semi-copulas, of properties of
stochastic dependence. In this perspective, we aim to discuss some intricate aspects of conceptual
character and to provide the readers with pertinent remarks from a Bayesian Statistics standpoint. In particular we will discuss the
role of extensions of dependence properties. \textquotedblleft
Archimedean\textquotedblright\ models have an important role in the present framework.

In the second part of the paper, the definitions of Kendall distribution and
of Kendall equivalence classes will be extended to semi-copulas and related
properties will be analyzed. On such a basis, we will consider the notion of
\textquotedblleft Pseudo-Archimedean\textquotedblright\ models and extend to
them the analysis of the relations between the ageing notions of  IFRA/DFRA-type and the
dependence concepts of  PKD/NKD.
\end{abstract}

\newpage
\section{Introduction}\label{sec:intro}

Let $\mathbf{X}\equiv\left( X_{1},...,X_{n}\right) $ be a vector of
non-negative random variables and denote by $\overline{F}_{\mathbf{X}%
}:\mathbb{R}_{+}^{n}\rightarrow\lbrack0,1]$ the \textit{joint survival
function} of $\mathbf{X}$:%
\[
\overline{F}_{\mathbf{X}}\left( x_{1},...,x_{n}\right) :=\mathbb{P}\left(
X_{1}>x_{1},...,X_{n}>x_{n}\right) .
\]

All along this note, if not explicitly mentioned  otherwise, $X_{1}%
,...,X_{n}$ are considered to be exchangeable and by $\overline{G}$ we denote
the one dimensional marginal survival function: for $j=1,...,n$ and $x\geq0$,
\[
\overline{G}\left( x\right) :=\mathbb{P}\left( X_{j}>x\right) .
\]

For simplicity's sake, $\overline{G}\left( \cdot\right) $ will be assumed
strictly decreasing and positive over the half-line $[0,+\infty)$.

When $X_{1},...,X_{n}$ are interpreted as \textit{lifetimes} of different
individuals, attention is often concentrated on the two different phenomena of
\textit{stochastic dependence} and of \textit{stochastic ageing}. A very rich
literature in applied probability has been devoted to this field and, from a
technical viewpoint, we remind that several different notions of dependence
and of ageing have been considered. Typically, such properties are described
in terms of inequalities involving comparisons between probabilities of
different events or conditional probabilities of a fixed event, given two
different information-states. See, e.g., references~\cite{BarPros1975}, \cite{Arjas1981},  \cite{Arjas-Norros1984}, \cite{Arjas-Norros1991},  \cite{Joe1997},     \cite{Spiz01}, \cite{Nelsen2006}, \cite{LaiXie2006},  \cite{ShaSha2007}, \cite{LiLi2013}, \cite{Spiz14}.

Dependence and ageing are strictly related one another and, at a time, they
are heavily affected by the actual state of information about $\left(
X_{1},...,X_{n}\right) $. We look at this circumstance from a Bayesian standpoint. In this respect
the relations between  dependence and ageing are relevant to understand the effects of those changes of information which do not destroy exchangeability.
\\

   Concerning the relations among such notions, in this paper we first review an approach that had been developed in the past years, for the case of exchangeable lifetimes. This approach is based on the role of the family of the level curves of joint survival functions in the description of ageing for inter-dependent lifetimes, from a Bayesian viewpoint. See~\cite{BarSpiz93},
    \cite{BasSpiz01}, \cite{Spiz01}, \cite{BasSpiz05a}. Some of the results and motivations contained in the previous references are  briefly recalled in
Sections~\ref{section:IFR-DFR}, \ref{section:exchangeable}, and \ref{section:sub-super-migrativity-IFR}.
See also~\cite{NapSpiz09} and~\cite{FosSpiz12} for further developments.
In such a context, the specific notion of semi-copula emerges as a rather natural extension of the notion of copula and several derivations hinge on the extension to semi-copulas of concepts of stochastic dependence.
In~\cite{BasSpiz05a} a conceptual method was introduced to single out the appropriate
notions of dependence and bivariate
ageing, \textquotedblleft corresponding"  to a fixed property of
one-dimensional ageing.
One aim of our work amounts to  illustrate this method. Our comments will point out some related  aspects, conceptual in nature, and propose a natural meaning to be given to the above term  \textquotedblleft corresponding" for the special class of dependence properties  defined in terms of Positive Dependence Orderings, see (\ref{C-POD-Pi}) and (\ref{G-B-POD-Pi}).
\\
\\

    In the afore-mentioned approach however some issues, concerning with properties of copulas and semi-copulas, still  require further analysis and developments.
     To this end we will introduce, in particular, a new definition of generalized Kendall distributions for semi-copulas, and, in Section~\ref{section:IFRA-PKD},  we will discuss the role of them
     and of related equivalence classes of semi-copulas. This study will also allow us to clarify intriguing aspects of that approach and to extend the results in~\cite{BasSpiz05a} about the relations among stochastic dependence and IFRA/DFRA-type concepts of ageing from Archimedean models to a larger class of models (see Proposition~\ref{prop:PKD-NAP-SPIZ}).
    \\
    More precisely, the paper has a structure as follows.
    In Section~\ref{section:IFR-DFR}, we concentrate attention on some simple aspects of the basic notions of Increasing Failure Rate (IFR) and Decreasing Failure Rate (DFR) probability distributions and present several related comments under the viewpoint of Subjective Probability and of Bayesian Statistics. Some results concerning relations between such notions and dependence properties of exchangeable lifetimes $X_1,...,X_{n}$ are recalled in Section~\ref{section:exchangeable}. Section~\ref{section:richiami}   is devoted to reviewing necessary notions about bivariate copulas, semi-copulas, dependence properties, Kendall distributions and bivariate ageing functions. We will also introduce the concept of generalized Kendall distributions for semi-copulas.  The relevant case of Archimedean copulas and semi-copulas is treated in details in Subsection~\ref{subsec:archimedean}.
    In  Section~\ref{section:sub-super-migrativity-IFR} we  concentrate attention on the ageing notions of IFR/DFR and review some specific results concerning with relations between them
    and corresponding concepts of positive dependence.
    In  Section~\ref{sec:path} we explain in details a conceptual method that was introduced in~\cite{BasSpiz05a} and that  leads to appropriate extensions of the arguments reviewed in the previous Section~\ref{section:sub-super-migrativity-IFR}. Such extensions concern other notions of dependence and univariate/bivariate ageing. Furthermore we pave the way to the extension, to be developed in Section~\ref{section:IFRA-PKD}, of the results
     about  the relations among stochastic dependence and IFRA/DFRA-type concepts of ageing. On this purpose, we  present definitions, properties, and mathematical results about equivalence classes of semi-copulas, which are defined in terms of generalized Kendall distributions.
     In Section~\ref{sec:conclusions}, we present a short discussion containing some comments, concluding remarks, and open problems.

\section{A brief review about one-dimensional IFR and DFR properties}\label{section:IFR-DFR}

Let $T$ be a non-negative random variable and let the symbol $\overline{G}$ to
be again used to denote the survival function of $T$.

As a basic concept and a paradigmatic notion of ageing we recall that $T$ is
\textit{Increasing Failure Rate} (IFR)\ when, for $t,s\geq0$,
\[
\mathbb{P}\left( T>s+t|T>t\right)
\]
turns out to be a non-increasing function in the argument $t$, for any fixed~$s$.
Similarly, $T$ is \textit{Decreasing Failure Rate} (DFR)  when
\[
\mathbb{P}\left( T>s+t|T>t\right)
\]
is a non-decreasing function of $t$. According to a common language,
the notion of IFR defines a concept of \textit{positive ageing}, whereas DFR
defines a concept of \textit{negative ageing}. If $T$ is
exponentially distributed then it is IFR and DFR, at a time, and the lack of
memory property of univariate exponential distributions is also seen as a
property of\textit{\ no-ageing}.

The arguments in this section will be concentrated on the notions of IFR and
DFR. This will constitute a basis for the brief discussion that will be
presented in the next Section~\ref{section:sub-super-migrativity-IFR}.

In the following result attention is preliminarily focused on different
characterizations of the concept of IFR distribution. The proof is almost
immediate and will be omitted. Details can be found, e.g., in~\cite{Spiz01}.
Before stating it, we recall the notion of Schur-Concavity (Schur-Convexity):
A function $W: \mathbb{R}_+\times\mathbb{R}_+ \rightarrow \mathbb{R}_+$ is Schur-concave (Schur-convex) iff
\begin{equation}\label{def-SCHUR-conc-convex}
\text{ for } 0 \leq x \leq y, \quad t\geq 0, \quad W(x+t,y) \geq (\leq) W(x,y+t).
\end{equation}
It is crucial in our discussione that Schur-Concavity (Schur-Convexity)  for a function $W(x,y)$ is a property of the level curves $\{(x,y)\; s.t.\; W(x,y)=c\}$.  For this and other properties of Schur-concave functions, see, e.g.,~\cite{MarshOlk1979}.

\begin{proposition}\label{prop:IFR-equiv}
The following conditions are equivalent
\begin{enumerate}
\item[(i)] $T$ is \emph{IFR}
\item[(ii)] $\overline{G}\left( \cdot\right) $ is log-concave
\item[(iii)] The   function $ (t_{1},t_{2})\mapsto \overline{G}\left( t_{1}\right)
\cdot$ $\overline{G}\left( t_{2}\right) $ is Schur-concave
\item[(iv)] For two i.i.d.\@ random variables $T_{1},T_{2}$, distributed according to
$\overline{G}\left( \cdot\right) $ one has, for given $s>0$ and for given
$0\leq t_{1}<t_{2}$,%
\begin{equation}
\mathbb{P}\left( T_{1}>t_{1}+s|T_{1}>t_{1},T_{2}>t_{2}\right) \geq
\mathbb{P}\left( T_{2}>t_{2}+s|T_{1}>t_{1},T_{2}>t_{2}\right) .
\label{BivIFR}%
\end{equation}

\end{enumerate}

\end{proposition}

\begin{remark}\label{rem:IFR-g/G}
In the case when a non-negative random variable $T$ has an absolutely
continuous one-dimensional marginal distribution, with a density probability
function denoted by $g(t)$, we can consider the \emph{failure rate function}%
\[
r(t):=-\frac{d}{dt}\log\overline{G}\left( t\right) =\frac{g(t)}{\overline
{G}\left( t\right) }.
\]
By item $(ii)$ of Proposition~\ref{prop:IFR-equiv}
the latter is, of course,  non-decreasing
 when~$T$
is~\emph{IFR} (whence the present terminology just arises).
\end{remark}

\begin{remark}\label{rem:IFR-univ}
Item $(iii)$ of~Proposition~\ref{prop:IFR-equiv} is a property of the level curves of the survival function $\overline{G}(t_1) \overline{G}(t_2)$. We notice furthermore that the
characterizations given in both items~$(iii)$ and~$(iv)$
 allow us to
express the univariate positive ageing notion of \emph{IFR} in terms of conditions
for the bivariate joint distribution of two i.i.d.\@ random variables.  These
observations will turn out to be relevant in our discussion.
\end{remark}

By suitably modifying  the statement of Proposition~2.1,
and also the above two Remarks,
 one can directly obtain corresponding statements that are valid for
the univariate, negative-ageing, concept of DFR. So far there is in fact a
complete symmetry between the two notions of IFR and DFR. On the contrary,
the following result  points out an aspect of lack of
symmetry between the two notions.

Let $\Xi$ be a set of indexes and, for the sake of notational
simplicity, let us consider a family of absolutely continuous survival
functions over the half-line $[0,+\infty)$ $\{\overline{G}_{\theta};\theta\in\Xi\}$ and let $\overline{G}$ be given by a mixture of the~$\overline{G}_{\theta}$'s.

\begin{proposition}\label{prop:DFR-mixture}
If $\overline{G}_{\theta}$ is \emph{DFR}, $\forall\theta\in\Xi$, then
also  any mixture $\overline{G}$ is \emph{DFR}.
\end{proposition}

\begin{remark}\label{rem:mixture-of-DFR-non-necess-IFR}
Generally a mixture $\overline{G}$ will not be \emph{IFR}, under the condition that
$\overline{G}_{\theta}$ is \emph{IFR}, $\forall\theta\in$ $\Xi$.
\end{remark}
\begin{remark}\label{rem:mixture-of-DFR-non-necess-IFR-bis}
It is easy to prove Proposition~\ref{prop:DFR-mixture}. As an immediate consequence, since the exponential distributions are \emph{DFR},  one obtains that a mixture of
exponential distributions is \emph{DFR}. Since the exponential distributions are \emph{IFR} as well, this property is a counterexample related
with the above Remark~\ref{rem:mixture-of-DFR-non-necess-IFR}.
\end{remark}

We now present an argument that allows one to understand both
the Proposition~\ref{prop:DFR-mixture} and the above Remarks from the point of view of Bayesian Statistics.

We set $\Xi
\equiv\mathbb{R}_{+}$ and consider,
along with $T$, a non-negative random variable $\Theta$ such that the joint distribution of $\left( \Theta,T\right) $ is absolutely
continuous with respect to the product of the two marginal distributions.
\\ In what follows $\Pi_{\Theta}$ denotes the marginal probability distribution of $\Theta$, and,
for $\theta\geq0$, $\overline{G}_{\theta}$ and $g_{\theta}$, denote the
conditional survival function and conditional density function of $T$ given
$  \Theta=\theta  $, respectively, namely%
\[
\overline{G}_{\theta}(t)=\int_{t}^{+\infty}g_{\theta}\left( s\right) ds.
\]
With this notation the failure rate function of the conditional distribution of $T$, given~$\Theta=\theta$,
is the ratio
\[
r_{\theta}(t)=\frac{g_{\theta}(t)}{\overline{G}_{\theta}(t)}.
\]

The one-dimensional survival function of $T$ and its probability density
function are respectively defined by the mixtures
\begin{equation}
\overline{G}\left( t\right) =\int_{0}^{+\infty}\overline{G}_{\theta}\left(
t\right)  d\Pi_{\Theta}\left(  \theta\right)  ;\quad g\left(  t\right)  =\int_{0}^{+\infty}g_{\theta}\left(  t\right)  d\Pi_{\Theta}\left(  \theta\right). \label{PredictiveSurvFnct&Density}%
\end{equation}

Consider now the failure rate function $r$ of $T$%
\[
r(t)=\frac{g\left(  t\right)  }{\overline{G}\left(  t\right)  }
=\frac{\int_{0}^{+\infty}g_{\theta}\left(  t\right)  d\Pi_{\Theta}\left(  \theta\right)}{\int_{0}^{+\infty}\overline{G}_{\theta}\left(  t\right)  d\Pi_{\Theta}\left(  \theta\right)  }
\]
and the conditional probability distribution of $\Theta$ given the event
$\left\{  T>t\right\}  $, which is defined by the equation%
\[
d\Pi_{\Theta}\left(  \theta|T>t\right)  =\frac{\overline{G}_{\theta}\left(
t\right)  d\Pi_{\Theta}\left(  \theta\right)  }{\int_{0}^{+\infty}\overline
{G}_{\theta}\left(  t\right)  d\Pi_{\Theta}\left(  \theta\right)  }.
\]
Thus we can write%
\begin{equation}
r(t)=\int_{0}^{+\infty}r_{\theta}(t)d\Pi_{\Theta}\left(  \theta|T>t\right)  .
\label{PredictiveFailureRate}%
\end{equation}

We can now compare the mixing measure $\Pi_{\Theta}\left(  \theta\right)  $
and the mixing measure $\Pi_{\Theta}\left(  \theta|T>t\right)  $, appearing in~(\ref{PredictiveSurvFnct&Density}) and~(\ref{PredictiveFailureRate}) respectively.

The latter measure does obviously depend on $t$, while the former does not.
This aspect is at basis of the argument aiming to justifying
Proposition~\ref{prop:DFR-mixture},  Remark~\ref{rem:mixture-of-DFR-non-necess-IFR}  and Remark~\ref{rem:mixture-of-DFR-non-necess-IFR-bis} .

In this respect we assume, for simplicity's sake, that $\Pi_{\Theta}$ is
absolutely continuous, denote its density function by $\pi_{\Theta}\left(
\theta\right)  $, and consider, for $0\leq t_{1}<t_{2}$, the ratio between the
conditional densities of $\Theta$, given the events $\left\{ T>t_{1}\right\}  $
and $\left\{  T>t_{2}\right\} $. By using the Bayes Formula, we obtain%
\begin{equation}\label{ratio-pi-theta}
\frac{\pi_{\Theta}\left(  \theta|T>t_{2}\right)  }{\pi_{\Theta}\left(
\theta|T>t_{1}\right)  }=\frac{\pi_{\Theta}\left(  \theta\right)  \overline
{G}_{\theta}\left(  t_{2}\right)  }{\pi_{\Theta}\left(  \theta\right)
\overline{G}_{\theta}\left(  t_{1}\right)  }\frac{\overline{G}\left(
t_{1}\right)  }{\overline{G}\left(  t_{2}\right)  }.
\end{equation}

As a substantial simplification, furthermore, we consider the case when
$r_{\theta}(t)$ is monotone w.r.t.\@ the variable $\theta$,  for any $t\geq0$: for example,
\begin{equation}
r_{\theta^{\prime}}(t)\leq r_{\theta^{\prime\prime}}(t), \quad \text{for any } \theta^\prime \leq \theta^{\prime\prime}.\label{HazRateOrdering}%
\end{equation}
Namely we consider the case when the ratio%
\[
\frac{\overline{G}_{\theta^{\prime\prime}}\left(  t\right)  }{\overline
{G}_{\theta^{\prime}}\left(  t\right)  }=\exp\left\{-\int_{0}^{t}\left[
r_{\theta^{\prime\prime}}(s)-r_{\theta^{\prime}}(s)\right]  ds\right\}
\]
is non-increasing as a function of $t$.  Under the assumption~(\ref{HazRateOrdering}), we separately consider now the cases when the IFR
property holds for all the conditional distributions of $T$ given $ \Theta=\theta  $ and when  the DFR property holds for all the
conditional distributions of $T$ given $  \Theta=\theta$.

For $0\leq t_{1}<t_{2}$, 
we have     that the
ratio $\frac{\overline{G}_{\theta}\left(  t_{2}\right)  }{\overline{G}%
_{\theta}\left(  t_{1}\right)  }$, and therefore the ratio in~(\ref{ratio-pi-theta}), is a non-increasing function of~$\theta$.

 This condition implies (see, e.g., \cite{ShaSha2007};
 see also Chapter~3 in~\cite{Spiz01})
that the conditional
distribution  $\Pi_{\Theta}\left(  \theta|T>t_{1}\right)  $ is stochastically
greater than $\Pi_{\Theta}\left(  \theta|T>t_{2}\right)  $, namely, for any
non-decreasing function $\delta\left(  \theta\right)  $, one has
\[
\int_{0}^{+\infty}\delta\left(  \theta\right)  \pi_{\Theta}\left(
\theta|T>t_{1}\right)\, d\theta  \geq\int_{0}^{+\infty}\delta\left(  \theta\right)
\pi_{\Theta}\left(  \theta|T>t_{2}\right) \, d\theta .
\]

Thus, in the conditionally DFR case, when, for any $\theta\geq0$, the functions $r_{\theta}(t)$ are
non-increasing w.r.t.\@ $t$, one obviously have that
$r(t)$ is non-increasing as well. Indeed
\begin{align*}
r(t_{1})&=\int_{0}^{+\infty}r_{\theta}(t_{1})\pi_{\Theta}\left(  \theta
|T>t_{1}\right) \, d\theta \geq\int_{0}^{+\infty}r_{\theta}(t_{1})\pi_{\Theta}\left(
\theta|T>t_{2}\right)\, d\theta
\intertext{and therefore, since $r_{\theta}(t_{1})\geq r_{\theta}(t_{2})$, one gets}& \geq\int_{0}^{+\infty}r_{\theta}(t_{2})\pi_{\Theta
}\left(  \theta|T>t_{2}\right) \, d\theta =r(t_2).
\end{align*}

On the contrary, in the conditionally IFR case, when, for any $\theta\geq0$, the functions $r_{\theta}(t)$ are
non-decreasing w.r.t.\@ $t$,   one is not allowed to
conclude that also $r(t)$ is non-decreasing since, in this case, $r_{\theta}(t_{1})\leq r_{\theta}(t_{2})$.

It is thus clear that, even assuming the IFR property for all the conditional
distributions of $T$ given $\Theta=\theta$, we cannot
generally conclude that the marginal distribution of $T$ is IFR as well
(notice that the latter conclusions would also stand valid under the condition
that $r_{\theta}(t)$ is decreasing w.r.t.\@ the argument $\theta$). In this
respect, see also arguments in~\cite{barlow1985bayes}.
 The present circumstance, that
the IFR property can go lost under the operation of making mixtures, can also
be understood as a \textit{Simpson-type paradox}.
(See Remark~\ref{rem:BivIFR-versus-MarginIFR} below)

\section{The case of exchangeable variables}\label{section:exchangeable}

In the items $(iii)$ and $(iv)$ of Proposition~\ref{prop:IFR-equiv},
the case of two i.i.d.\@ random times was
considered. Focusing attention on the vector $\mathbf{X}\equiv\left(
X_{1},...,X_{n}\right) $ of non-negative \textit{exchangeable} (non-independent) random
variables, we notice that it still makes sense to consider for $X_{1},...,X_{n}$ a condition such as in~(\ref{BivIFR}):   namely, for $i\neq j$, for $0\leq x<y$, and $t\geq 0$,
\begin{equation}
\mathbb{P}\left( X_{i}>x+t|X_{i}>x,X_{j}>y\right) \geq\mathbb{P}\left(
X_{j}>y+t|X_{i}>x,X_{j}>y\right) . \label{BivIFRforX}%
\end{equation}

Actually, in a subjective-probability or Bayesian framework, the latter can be interpreted as a \textit{bivariate condition of
positive ageing} (see~\cite{Spiz92}, \cite{BarMen92}, \cite{BarSpiz93},
\cite{BasSpiz99}, \cite{Spiz01}, \cite{LaiXie2006}). As remarked above, (see~(\ref{BivIFR})  in Proposition~\ref{prop:IFR-equiv}), such a condition is
equivalent to the IFR property of the marginal survival function~$\overline
{G}$, when $X_{1},...,X_{n}$ are i.i.d.\@ variables.
The opposite inequality
\begin{equation}
\mathbb{P}\left( X_{i}>x+t|X_{i}>x,X_{j}>y\right)  \leq \mathbb{P}\left(
X_{j}>y+t|X_{i}>x,X_{j}>y\right) . \label{BivDFRforX}%
\end{equation}
is, on the contrary, equivalent to the DFR property of $\overline{G}$ and can therefore be interpreted as a \textit{bivariate condition of negative ageing}.
\\

The above inequalities~(\ref{BivIFRforX}) and~(\ref{BivDFRforX}) are conditions on   the joint bivariate survival function  of
any pair $\left( X_{i},X_{j}\right) $,   with  \mbox{$1\leq i\neq j\leq n$}, and  we will denote the latter by
\begin{equation}\label{F-2}
\overline{F}^{(2)}(x,y)=\mathbb{P}\left( X_{i}>x,X_{j}%
>y\right) ;\quad x,\,y\geq0.
\end{equation}
In order to highlight that~(\ref{BivIFRforX}) and~(\ref{BivDFRforX})  are properties of  positive/negative of bivariate ageing, it will be convenient to say that $\overline{F}^{(2)}(x,y)$ is \emph{Bayesian bivariate Increasing/Decreasing Failure Rate}, abbreviated to Bayesian biv-IFR/Bayesian biv-DFR.\\

Let us now consider the case of conditionally independent, identically distributed, variables $X_{1},...,X_{n}$.
What can be said for this case?

Let $\Theta$
 be a $\Xi$-valued random parameter (with $\Xi
\subseteq\mathbb{R}^{d}$, say), with probability distribution
$\Pi_{\Theta}$
and let $X_{1},...,X_{n}$ be conditionally independent
given $\Theta$,
with a conditional one-dimensional survival function
$\overline{G}(\cdot|\theta)$ for $\theta\in\Xi$, namely
\[
\overline{F}_{\mathbf{X}}\left( x_{1},...,x_{n}\right) =
\int_\Xi\overline{G}(x_{1}|\theta)\cdot...\cdot\overline{G}(x_{n}|\theta)
d\Pi_\Theta(\theta).
\]
In such a case, the condition
\[
\overline{G}(\cdot|\theta)\,\text{ is IFR }\quad \forall\;\theta\in\Xi
\]
implies that the
bivariate condition of
positive ageing~(\ref{BivIFRforX}) holds true. However  (see Remarks~\ref{rem:mixture-of-DFR-non-necess-IFR}, \ref{rem:mixture-of-DFR-non-necess-IFR-bis}, and subsequent arguments)  such a condition does not
imply the IFR property of the marginal survival function
\[
\overline{G}(x)= \int_{\Xi}\overline{G}(x|\theta)d\Pi_{\Theta}(\theta),
\]
which would result, for $t>0$ and $0\leq x<y$, in the inequality%
\begin{equation}
\mathbb{P}\left( X_{i}>x+t|X_{i}>x\right) \geq\mathbb{P}\left(
X_{j}>y+s|X_{j}>y\right) . \label{MarginIFR}%
\end{equation}

The condition
\[
\overline{G}(\cdot|\theta)\;\text{ is DFR }\quad \forall\;\theta\in\Xi
\]
implies that both the
 condition that $\overline{G}$ is   DFR  and the
condition of bivariate negative ageing~(\ref{BivDFRforX})
hold  true.

\begin{remark}\label{rem:BivIFR-versus-MarginIFR}
Notice that the comparison in (\ref{BivIFR}) is established between the
conditional probabilities of two different events given a same conditioning
event. In (\ref{MarginIFR}), on the contrary, we compare two conditional
probabilities containing two different conditioning events. The inequality
(\ref{MarginIFR}) is not implied then by the assumption
\[
\mathbb{P}\left( X_{i}>x+t|X_{i}>x;\theta\right) \geq\mathbb{P}\left(
X_{j}>y+t|X_{j}>y;\theta\right)
\]
as we had, e.g., shown above, for the case when $\Xi=\mathbb{R}_{+}$ and
$\Theta|\left( X_{1}>x\right) $ is stochastically larger than $\Theta
|\left( X_{1}>y\right) $, for $0\leq x<y$. This conclusion can be looked at
as a Simpson-type Paradox in the sense of (see~\cite{ScarSpiz99}).
\end{remark}
\bigskip

As it is well known, dating back to the original work by de Finetti (see, e.g.,~\cite{Def1937}),
the other cases of exchangeability different from those of conditional  independent and identical distribution, are those of \textit{finite  exchangeability}.

What about the relation between the univariate and bivariate properties of
ageing in the case when, alternatively to conditional independence, we assume
$X_{1},...,X_{n}$ to be \textit{finitely exchangeable}?

Such relations are generally influenced by  the type of stochastic dependence among $X_{1},...,X_{n}$.
The property of conditional independence does, in any case, imply some
sort of positive dependence among $X_{1},...,X_{n}$. At least, positive
correlation between $X_{1},X_{j}$ for $1\leq i\neq j\leq n$, as is very well-known.

In the case when $X_{1},...,X_{n}$ are finitely exchangeable, the relations
between the ageing properties of $\overline{G}$ and the conditions
(\ref{BivIFRforX}), (\ref{BivDFRforX}) may be a bit more involved. Actually,
in such a case, one can meet different types of (positive or negative)
stochastic dependence (see, e.g.,~\cite{Spiz01}) and the marginal survival function $\overline{G}$ can
still be a mixture of given survival functions, but some of the coefficients
of the mixture may be negative. (See, e.g., \cite{kerns-et-al2006}, \cite{Jan-Konst16repres}, \cite{Leon16}).
\\
For our purposes, however, it is not really relevant to distinguish between
finite or infinite exchangeability. Rather, we look at the actual properties
of stochastic dependence of bivariate distributions, i.e., of the joint survival function~$\overline{F}^{\left( 2\right) }$, defined in~(\ref{F-2}).
The rest of the paper
will be then devoted to showing some results relating stochastic dependence and bivariate ageing properties. Preliminarily, on this purpose it is convenient to present a brief review of technical definitions and related properties concerning bivariate survival models as in (\ref{F-2}). This will be done in the next section, before continuing our analysis in the subsequent sections.

\section{A review about bivariate copulas and ageing functions} \label{section:richiami}

For our convenience, and to fix the notation, we start this section by just recalling  well-known
facts about bivariate copulas.
We recall  that a \emph{bivariate copula} is a function $C:\left[  0,1\right]^{2}\rightarrow\left[  0,1\right]  $ such
that
\begin{align}
C\left(  0,v\right)   &  =C\left(  u,0\right)  =0,\quad0\leq u,v\leq;
1\label{eq:copuno}\\
C\left(  1,v\right)   &  =v,\quad C\left(  u,1\right)  =u,\quad0\leq
u,v\leq1;\label{eq:copdue}\\
C\left(  u,v\right)  & \quad\mathrm{is\ increasing \ in \ each \ variable;}
\label{eq:copquattro}
\\
C\left(  u,v\right)  &+C\left(  u^{\prime},v^{\prime}\right) -C\left(
u,v^{\prime}\right)  -C\left(  u^{\prime},v\right) \geq0,
\label{eq:copcinque}%
\\&\text{${}$\quad for all $ 0\leq u\leq
u^{\prime}\leq1,\quad0\leq v\leq v^{\prime}\leq1$.}\notag
\end{align}
In other words a bivariate copula is the restriction to $[0,1]^2$ of a joint distribution function for a pair a random variables $U, V$, uniformly distributed in $[0,1]$.
\\
From now on the term bivariate will be  generally dropped.
\\
Three special copulas are the following ones:
\begin{align*}
\text{the \emph{independence} copula:}\quad \Pi(u,v)&=uv;
\\
\text{the \emph{maximal} copula:}\quad M(u,v)&= u\wedge v;
\\
\text{the \emph{minimal} copula:}\quad W(u,v)&=\max\big(1-(u+v),0\big).
\end{align*}
\\
We remind that any bivariate distribution function  $F(x,y)$, with marginals $F_X(x), \, F_Y(y)$ can be written as
$$
F(x,y)= C(F_X(x), F_Y(y)),
$$
where $C$ is a copula. When $F_X$, $F_Y$ and $F$ are continuous, then the copula is unique and we refer to it as the corresponding \emph{connecting copula}.
\\
Moreover, under the same conditions, the bivariate survival function  $\overline{F}(x,y)$, with survival marginals $\overline{F}_X(x), \, \overline{F}_Y(y)$ can be written as
\begin{equation}\label{Fbar-con-C-hat}
\overline{F}(x,y)= \widehat{C}(\overline{F}_X(x), \overline{F}_Y(y)),
\end{equation}
where $\widehat{C}$ is the corresponding \textit{two-dimensional survival copula}.  Furthermore looking at $C$ and $\widehat{C}$ as joint distribution functions of the pairs  $(U,\,V)$  and $(\widehat{U},\,\widehat{V})$, respectively, one can write
\begin{align}\label{def:U-V}
U=F_X(X),\; V=F_Y(Y), \quad \widehat{U}=\overline{F}_X(X)=1-U,\; \widehat{V}=\overline{F}_Y(Y)=1-V.
\end{align}
The survival copula $\widehat{C}$ is linked to the connecting copula  $C$ by means of the following relation
$$
\widehat{C}(u,v)=u+v-1+C(1-u,1-v).
$$
\indent The concept of copula is relevant in the description of dependence properties among random variables (see, e.g., \cite{Joe1997}, \cite{Nelsen2006},  \cite{DurSemBook2016}). Often a dependence property  for a joint bivariate distribution, is equivalent to the same dependence property of the   connecting copula.
Obviously, properties of dependence may be assessed for the survival copula as well. It may happen that assessing a dependence property on the survival copula or on the connecting copula gives rise to different conditions for the joint distribution (see in particular  Remarks~\ref{rem:LTD} and~\ref{rem:PKD}, below).
\\

 We now pass to recalling a few relevant definitions of dependence properties.
We start with the \textit{Positive Quadrant Dependence}~(PQD) property, i.e., the events $\{Y\leq y\}$ and $\{X\leq x\}$, are positive correlated for all $x$ and $y$:
\begin{align*}
C(u,v)\geq  \Pi(u,v)=  uv \quad &\Leftrightarrow \quad F(x,y)\geq F_X(x)\, F_Y(y),
\\
\Updownarrow\quad \qquad \qquad \qquad  &{} \qquad \qquad \qquad  \Updownarrow
\\
\widehat{C}(u,v)\geq  \Pi(u,v)=  uv \quad &\Leftrightarrow \quad \overline{F}(x,y)\geq \overline{F}_X(x)\, \overline{F}_Y(y).
\end{align*}
Another dependence property to be recalled is the \emph{Stochastic} \emph{In\-creasing\-ness}~(SI) in $X$:
$$
x\mapsto \mathbb{P}(Y>y|X=x) \quad \text{is  increasing,}
$$
which is equivalent to the following properties of $C(u,v)$ and $\widehat{C}(u,v)$
$$
u\mapsto \mathbb{P}(V>v|U=u) \; \text{is increasing,}\quad \Leftrightarrow \quad  u \mapsto \mathbb{P}(\widehat{V}>v|\widehat{U}=u) \; \text{is increasing.}
$$
A further dependence property is the condition: $Y$ is  \emph{Left Tail Decreasing} (LTD) in $X$, i.e., the function $x \mapsto \mathbb{P}(Y\leq y| X\leq x)$ is decreasing:
\begin{align*}
\frac{F(x^\prime,y)}{F_X(x^\prime)}& \leq \frac{F(x,y)}{F_X(x)} \quad \text{for all $x\leq x^\prime$, and for all $y$,}
\\
&\Updownarrow
\\
\frac{C(u^\prime,v)}{u^\prime} & \leq \frac{C(u ,v)}{u }, \quad \text{for all $0< u\leq u^\prime\leq 1$, and for all $0\leq v\leq 1 $},
\end{align*}
i.e., $V$ is LTD in $U$, where $U$ and $V$ have been defined in~(\ref{def:U-V}).
\begin{remark}\label{rem:LTD}
The corresponding   property for the survival copula $\widehat{C}$, i.e., $\widehat{V}$ is \emph{LTD} in $\widehat{U}$,
namely
$$
\frac{\widehat{C}(u^\prime,v)}{u^\prime}  \leq \frac{\widehat{C}(u ,v)}{u }, \quad \text{for all $0< u\leq u^\prime\leq 1$, and for all $0\leq v\leq 1 $},
$$
is instead equivalent to the condition that $Y$ is \emph{Right Tail Increasing} (\emph{RTI}) in $X$, i.e., the function $x\mapsto \mathbb{P}(Y> y| X> x)$ is increasing.
\end{remark}

We are now going to recall further dependence properties that are specially relevant for what follows.

We start with the so called \textit{Supermigrativity} property: for an arbitrary copula $D$, i.e.,
\begin{equation}
D\left( us,v\right) \geq D\left( u,sv\right) , \label{SuprMigrtv}%
\end{equation}
for $0\leq v\leq u\leq1$ and $s\in\left( 0,1\right) $. 

This condition,  applied to the survival copula $\widehat{C}$,  has emerged to describe the
Schur-concavity of $\overline{F}$ (see~\cite{BasSpiz05a}, see also Proposition~\ref{lemma-Equiv-SCHUR},  below).
The term \textit{Supermigrativity} has been
coined in~\cite{DurGhis09}.
One says that a copula $D$ is \textit{submigrative} when the direction of inequality is inverted in~(\ref{SuprMigrtv}):
\begin{equation}
D\left( us,v\right) \leq D\left( u,sv\right) , \label{SubMigrtv}%
\end{equation}
for $0\leq v\leq u\leq1$ and $s\in\left( 0,1\right) $.

In a sense, Supermigrativity can be
seen as a property of positive dependence: In particular it implies the
PQD property.
\begin{example} \label{example:Comonotone-supermigrative}
The  extreme case of positive dependence is the one of  \textquotedblleft comonotone" dependence:
$$ \mathbb{P}(X=Y)=1$$
In such a case the survival copula coincides with the  maximal  copula:
$$
\widehat{C}(u,v)=M(u,v)= u\wedge v,
$$
and is obviously supermigrative.
\end{example}

In order to highlight that Super/Submigrativity are   properties of positive/negative dependence   it will be convenient to use also the terms \emph{Positive/Negative Migrativity Dependence}, abbreviated to PMD/NMD.
\begin{remark}\label{rem:LTD--SuperMigr-ARCH}
As pointed out in \cite{BasSpiz05a} (see Proposition 6.1 therein), Supermigrativity (or equivalently \emph{PMD} property) does coincide with
the \emph{LTD} property in the case of an Archimedean copula (a brief review of Archimedean copulas will be given next).
Several other features  related with Supermigrativity have been analyzed in~\cite{DurGhis09}).
\end{remark}

The last  dependence property to be recalled here is the \emph{Positive Kendall Dependence} (PKD) property (see Definiton~\ref{def:PKD}, below). This property will have a relevant role for our results in Section~\ref{section:IFRA-PKD} and is  connected with the
  \emph{Kendall distribution function
associated to} $F$ (see, e.g., \cite{NelAl03})
\begin{equation}
K(t):=\mathbb{P}(F(X,Y)\leq t)= \mathbb{P}(C(U,V)\leq t). \label{Kendall-F}
\end{equation}
We point out that $K(t)$ depends only on the connecting copula $C$, and we will also use the notation $K_C(t)$. We also recall that $K_C(t)\geq t$ in $[0,1]$, as is easily checked.

In particular for the independent copula $\Pi(u,v)=uv$, one has
$$
 K_\Pi(t)= t - t \log(t).
 $$

 In analogy with the Kendall distribution one can consider (see~\cite{NapSpiz09}) the \textit{upper-orthant Kendall
distribution} associated to~$\overline{F}$, i.e.,
\begin{equation}
\widehat{K}(t):=\mathbb{P}(\overline{F}(X,Y)\leq
t)=\mathbb{P}(\widehat{C}(\widehat{U}, \widehat{V})\leq t). \label{UOBipitDistr}
\end{equation}

In other words $\widehat{K}(t)=K_{\widehat{C}}(t)$.
\begin{remark}\label{rem:mu-C}
Note that
$$
K_C(t)= \mu_C \big(\{(u,v) \in [0,1]^2:\; C(u,v)\leq t\} \big).
$$
where  $\mu_C$ is the probability measure on $[0,1]^2$ such that
$$\mu_C((u ,u^\prime]\times  (v ,v^\prime])= C\left(  u,v\right)  +C\left(  u^{\prime},v^{\prime}\right) -C\left( u,v^{\prime}\right)  -C\left(  u^{\prime},v\right).$$
\end{remark}

The above remark suggests an alternative way to compute $K_C(t)$. This will be done in the following Lemma~\ref{lemma:KC-con-sums}, the proof of which is almost immediate and will be omitted.
\\
Let
$$\{u_0=0 \leq u_{1}\leq \cdots \leq u_i \leq u_{i+1}\leq \cdots \leq u_n\}$$
 be a finite partition of $(0,1]$, and denote by $ \mathcal{P}$ the class of all finite partitions of~$(0,1]$.
 Set furthermore
$$
C^{-1}_u(t):= \sup\{v:\; C(u,v)\leq t\},
$$
the generalized inverse of $v\mapsto C(u,v)$.

\begin{lemma}\label{lemma:KC-con-sums}
Let $C(u,v)$ be a copula, then
\begin{align}
K_C(t)&=\sup_{\mathcal{P}} \Big\{\sum_{i\in I} [C(u_{i+1},v_i)-C(u_i,v_i)], \; \text{with }\; v_i=C^{-1}_{u_{i+1}}(t)\Big\}
\end{align}
\end{lemma}

\begin{definition}[Positive (Negative) Kendall Dependence]\label{def:PKD}
A copula $D$ is \emph{Positive Kendall Dependent} (\emph{PKD}) when
$$
K_D(t) \geq K_\Pi(t)= t - t \log(t), \quad t\in (0,1),
$$
while $D$ is \emph{Negative Kendall Dependent} (\emph{NKD}) when the reverse inequality holds
Two random variables $X,\,Y$ are Positive (Negative) Kendall Dependent  (\emph{PKD/NKD}) when their connecting copula $C$ is \emph{PKD/NKD}.
\\
Two random variables $X,\,Y$ are \emph{Positive (Negative) upper-orthant Kendall Dependent}  (\emph{PuoKD/NuoKD}) when their survival copula $\widehat{C}$ is \emph{PKD/NKD}.
\\
\end{definition}

\begin{remark}\label{rem:PKD}
We notice that the  \emph{PKD} property for the connecting copula $C$ and for the survival copula $\widehat{C}$ give rise to different dependence properties for the random variables $X$, $Y$.
\end{remark}

In the case of exchangeability,  a  survival function  $\overline{F}(x,y)$ takes the form
\begin{equation}\label{Fbar-con-C-hat-exch}
\overline{F}(x,y)= \widehat{C}(\overline{G}(x), \overline{G}(y)),
 \end{equation}
 where $\widehat{C}$ is an exchangeable copula and $\overline{G}$ denotes the common marginal survival function. Restricting our attention to this case we also associate  to $\overline{F}(x,y)$  a
function that describes the family of its level curves, according to the following definition (see~\cite{BasSpiz01}).
\begin{definition}[Bivariate ageing function] \label{def:Biv-ageing}
The function
$$B:\left[ 0,1\right]^2\rightarrow\left[ 0,1\right] ,$$  defined by
\begin{equation}\label{def:B}
B(u,v):=\exp\{-\overline{G}^{-1}\left( \overline{F}
\left( -\log u,-\log v\right) \right) \},
\end{equation}
is called bivariate ageing function of $\overline{F}(x,y)$.
\end{definition}
Taking into account the expression~(\ref{Fbar-con-C-hat-exch}) of the survival function we can also rewrite~(\ref{def:B}) as
\begin{equation}
B(u,v)=\exp\{-\overline{G}^{-1}\left(  \widehat{C}\big(\overline{G}(-\log u),\overline{G}(-\log
v)\big)\right)
\}. \label{DefB-tramite-hatC-e-G-bar}
\end{equation}
By setting
\begin{align}\label{def:gamma}
\gamma (u)&= \exp\{-\overline{G}^{-1}(u) \},
\intertext{whence}
\gamma^{-1} (z)&= \overline{G}(-\log(z)),\label{def:gamma-a-1}
\intertext{we can also write}
B(u,v)&=\gamma\big(\widehat{C}(\gamma^{-1}(u),\gamma^{-1}(v) \big).\label{def:B-con-C-hat-e-gamma}
\end{align}
From~(\ref{def:B}) one immediately obtains
\begin{equation}\label{F-2--Gbar-lnB}
\overline{F}
(x,y)= \overline{G}\big(-\log \big( B(e^{-x}, e^{-y})\big) \big).
\end{equation}
Any other  survival function $\overline{M}(x,y)$ sharing  with $\overline{F}(x,y)$ the same family of level curves, must also share the same  ageing function $B$. Therefore $\overline{M}(x,y)$ must be of the following form
\begin{equation*}
\overline{M}(x,y)= \overline{H}\big(-\log \big( B(e^{-x}, e^{-y})\big) \big),
\end{equation*}
for a marginal survival function $\overline{H}$. Nevertheless, for  arbitrary $\overline{H}$, it is not guaranteed that $\overline{H}\big(-\log \big( B(e^{-x}, e^{-y})\big) \big)$ is a bona-fide survival function. In~\cite{NapSpiz09} (see Theorem~2, therein),   conditions on $\overline{H}$ have been given to guarantee that  $\overline{H}\big(-\log \big( B(e^{-x}, e^{-y})\big) \big)$ is actually a survival function (see Remarks~10 and~24 therein).
\\
Furthermore, like it happens for any copula, $B$ is a $\left[ 0,1\right] $-valued function,
defined over $\left[ 0,1\right]^2$,  increasing in each variable, and it is such
that
\begin{equation}\label{def:semicopula-marg}
B(w,1)=B(1,w)=w, \quad B(w,0)=B(0,w)=0,   \forall\; w\in\left[ 0,1\right] .
\end{equation}
However the function $B$ is not always a copula. A bivariate copula, in fact, must
have the properties of a bivariate probability distribution
function. Whereas, on the contrary, examples of ageing functions $B$ can be given such that%
\begin{equation}\label{def:semicopula}
B(u^{\prime},v^{\prime})-B(u,v^{\prime})-B(u^{\prime},v)+B(u,v)<0
\end{equation}
for some values $0<u<u^{\prime}<1,0<v<v^{\prime}<1$.
\\
Nevertheless,  $B$ turns out to be a copula in some special cases.
\begin{example}\label{example:Schurconstant}
Let us consider the remarkable case of  \emph{Schur-constant} $\overline{F}(x,y)$ survival functions with marginal $\overline{G}$, i.e.,
$$
\overline{F}(x,y)= \overline{G}(x+y),
$$
(see~\cite{Spiz01}).
In such a case the function $B$ does coincides with the product copula $\Pi(uv)=uv$.\\
\end{example}

 In the limiting case of  \textquotedblleft comonotone" dependence, (see Example~\ref{example:Comonotone-supermigrative})
 one has
$$
	B(u,v)=\widehat{C}(u,v)=M(u,v)=u\wedge v.
$$
Thus  $B(u,v)$ is a supermigrative copula.
\\

The term \emph{semi-copula} has been proposed to designate functions which, like the ageing function $B$ above, are increasing in each variable, satisfy the margin conditions~(\ref{def:semicopula-marg}), and are such that the inequality in~(\ref{def:semicopula}) may hold for some values $0<u<u^\prime<1$,  $0<v<v^\prime<1$. For a more formal definition of semi-copula, basic properties, extensions and technical details about semi-copulas and ageing functions, also in a
multivariate context, see in particular, the papers~\cite{BasSpiz01},
\cite{BasSpiz03}, \cite{BasSpiz05a}, \cite{DurSem05-K}, \cite{DurSem05-I}, \cite{NapSpiz09},
\cite{DurSpiz10}, \cite{DurFosSpiz10}, and the book~\cite{DurSemBook2016} with
references cited therein. We point out that the class of   bivariate ageing functions is strictly contained in the one of semi-copulas (see, e.g.,~\cite{DurSemBook2016}).
\\
For our purposes, we only need to remind that it is useful to extend to
semi-copulas definitions of stochastic dependence that have been formulated
for copula functions.
In particular, it is immediate to define the property of Supermigrativity (Submigrativity) for any semi-copula $S$:
We say that $S$ is supermigrative (submigrative)  iff
\begin{equation}\label{Super-Sub-migrat-B}
 S(us, v) \geq (\leq) S(u,sv), \quad \text{whenever } 0\leq v \leq u \leq 1, \quad s \in (0,1).
 \end{equation}
Further dependence properties for semi-copulas, still relevant for our arguments, are the PKD/NKD properties. Giving this definition requires   an appropriate extension of the concept of Kendall distribution to semi-copulas. An extension had been  introduced for a special class of bivariate ageing functions in~\cite{NapSpiz09}. The latter extension
 can also be seen as a particular case of generalized Kendall distributions for semi-copulas, that  we introduce next, in analogy with Lemma~\ref{lemma:KC-con-sums}.
  \begin{definition}\label{Def:Kendall-semicopula}
  Let $S: \, [0,1]^2 \rightarrow [0,1]$ be a semi-copula, and set
$$
S^{-1}_u(t):= \sup\{v:\; S(u,v)\leq t\},
$$
the generalized inverse of $v\mapsto S(u,v)$,  with the convention that $\sup(\emptyset)=0.$
\\ The \emph{generalized Kendall distribution} associated to $S$ is the function \mbox{ $K_S:\; [0,1] \rightarrow \mathbb{R}^+$,} defined by
\begin{align}\label{eq:KS-sup}
K_S(t)&:=\sup_{\mathcal{P}} \Big\{\sum_{i\in I} [S(u_{i+1},v_i)-S(u_i,v_i)], \; \text{with }\; v_i=S^{-1}_{u_{i+1}}(t)\Big\},
\end{align}
where $ \mathcal{P}$ is the class of all finite partitions of~$[0,1]$, of the form $\{u_i, \, i\in I\}$, such that $u_0=0$ and  $u_i\leq u_{i+1}$.
  \end{definition}
  \begin{example}\label{example:KS_phi}
  Let $S=S_\varphi$, with $\varphi(t)$ strictly decreasing, differentiable, and such that $\varphi(1)=0$. Furthermore assume that $\varphi^\prime(t)>0$ in $(0,1)$.  Then
  $$
  K_S(t)= t-\frac{\varphi(t)}{\varphi^\prime(t)}.
  $$
  When $\varphi(t)$ is also convex, then $S$ is a copula and  this result is well known. Moreover, by Lemma~\ref{lemma:KC-con-sums}, when $S$ is a copula $K_S(t)$  can be computed by~(\ref{eq:KS-sup}), i.e., it holds
  $$
 t-\frac{\varphi(t)}{\varphi^\prime(t)}= \sup_{\mathcal{P}} \Big\{\sum_{i\in I} [S_\varphi(u_{i+1},v_i)-S_\varphi(u_i,v_i)], \; \text{with }\; v_i=(S_\varphi)^{-1}_{u_{i+1}}(t)\Big\}.
  $$
   The general case can be deduced observing that  the latter equality  also holds when $\varphi$ is not convex.
  \end{example}
We notice that, with this new definition, the generalized Kendall distribution has the properties
 \begin{equation}\label{eq:Ks-geq-t}
K_S(t)\geq t, \quad t \in [0,1],  \qquad K_S(0)=0,\qquad K_S(1)=1,
\end{equation}
 for any semi-copula $S$ as it happens for copulas. Indeed, first of all,  w.l.o.g., we may consider only partitions containing $t$ as an element and observe that, when $u_{i+1}\leq t$ then $v_i=S^{-1}_{u_{i+1}}(t) =1$, so that $S(u_{i+1},v_i)-S(u_i,v_i)=u_{i+1}-u_i$. Therefore 
\begin{align}\label{eq:KS-sup-t+}
K_S(t)&:=t+\sup_{\mathcal{P}} \Big\{\sum_{ \underset{u_{i+1}> t}{i\in I}} [S(u_{i+1},v_i)-S(u_i,v_i)], \; \text{with }\; v_i=S^{-1}_{u_{i+1}}(t)\Big\},
\end{align}
and~(\ref{eq:Ks-geq-t}) follows  by observing that the increments $S(u_{i+1},v_i)-S(u_i,v_i)$ are non-negative. The other two properties are obvious.
 \\
Furthermore we point out that, when $S$ is a Lipschitz semi-copula, i.e.,
$$
|S(u,v)-S(u^\prime, v^\prime)| \leq L_S \big( |u-u^\prime|+  |v-v^\prime|\big),
$$
 then
  $$
  K_S(t)\leq L_S \quad \text{in $[0,1]$.}
  $$
  Indeed in latter case one has  $0\leq S(u_{i+1},v_i)-S(u_i,v_i)\leq u_{i+1}-u_i$.
   In particular, when~$L_S=1$ , i.e., $S$~is a quasi-copula (see, e.g., \cite{DurSemBook2016})  then $K_S(t)\leq 1$.
 \\ Finally we observe that, when $S$ is not a copula,
  $K_S(t)$ may not be a probability distribution function as shown by the following example.
 \begin{example}
 Let $S=S_\varphi$, with
 $$
 \varphi(t)=\cos(\tfrac{\pi}{2}\, t),
 $$
 i.e., by taking into account Example~\ref{example:KS_phi},
 $$
 S_\varphi(u,v):= \arccos\big( \tfrac{2}{\pi}\, [\cos(\tfrac{\pi}{2}\, u) + \cos(\tfrac{\pi}{2}\, v)]\big).
 $$
 Then
 $$
K_S(t)=t-\tfrac{2}{\pi}\,\frac{\cos(\tfrac{\pi}{2}\, t)}{-\sin(\tfrac{\pi}{2} \,t) }=t+\tfrac{2}{\pi}\,\frac{\cos(\tfrac{\pi}{2}\, t)}{\sin(\tfrac{\pi}{2} \,t) },\quad t\in (0,1),
$$
is not a distribution function, since, in particular,  $K_S(t)$ converge to $\infty$ as $t$ goes to $0$.
 \end{example}

 In Section~\ref{section:IFRA-PKD}  we provide a formula for the generalized Kendall distribution of an ageing function $B(u,v)=\gamma\big(\widehat{C}(\gamma^{-1}(u),\gamma^{-1}(v) \big)$ in terms of $K_{\widehat{C}}(t)$ and of $\gamma(t)$ (see Proposition~\ref{prop:K-di-C-hat---K-di-B}).

\subsection{Archimedean copulas and semi-copulas}\label{subsec:archimedean}
 One important class of copulas is the one of Archimedean copulas:\\
Let $\phi:(0,1]\rightarrow\lbrack0,+\infty)$ be a continuous, convex, and
decreasing function and denote by $C_{\phi}$ the bivariate \textit{Archimedean
copula} with \emph{additive generator}~$\phi$. Namely, for $0\leq u,v\leq1$, set%
\[
C_{\phi}(u,v):=\phi^{-1}\left[ \phi(u)+\phi(u)\right].
\]
It is convenient, though not strictly necessary, to assume~$\phi$ strictly decreasing and such that
\[
\lim_{u\rightarrow0^{+}}\phi\left( u\right) =+\infty,\quad \phi\left( 1\right)
=0,
\]
so that the function~$\phi^{-1}$ can be identified with a one-dimensional
survival function of the type that we are considering here.
\\
Two basic examples of Archimedean copulas are the independent copula  $\Pi(u,v)=uv$, where $\phi(u)=-\log(u)$, and the survival  copula of a Schur-constant survival function as in Example~\ref{example:Schurconstant}, where $\phi(u)= \overline{G}^{-1}(u)$.

Similarly, we say that a semi-copula $S$ is \textit{Archimedean} when it has the form
\[
S_{\varphi}(u,v):=\varphi^{-1}\left[ \varphi(u)+\varphi(u)\right],
\]
where $\varphi:(0,1]\rightarrow\lbrack0,+\infty)$ is a continuous decreasing function, not necessarily convex. We maintain the assumptions
\[
\lim_{u\rightarrow0^{+}}\varphi\left( u\right) =+\infty,\quad \varphi\left( 1\right)
=0.
\]

Since $C_{\phi}$ is a copula, and $S_{\varphi}$ is a semi-copula it makes sense to consider dependence properties
of them. Since $\phi^{-1}$  and $\varphi^{-1}$ are, technically,   one-dimensional survival functions,
it makes sense to consider ageing properties of them. On the other hand, the inverse of any one-dimensional survival function~$\overline{G}$, if it exists, can be seen as the generator of an Archimedean semi-copula $S_{G^{-1}}$, and it makes sense to consider the bivariate ageing of this semi-copula.

In particular, to the marginal survival function $\overline{G}$ of an exchangeable bivariate survival function $\overline{F}(x,y)$,  it will be convenient to associate the Archimedean semi-copula
\begin{equation}\label{def:copula-Schur-const}
S_{\overline{G}^{-1}}(u,v)= \overline{G}\big(\overline{G}^{-1}(u)+\overline{G}^{-1}(v)\big).
\end{equation}
\begin{remark}\label{rem:Schurconstant}
When $\overline{G}$,  and then $\overline{G}^{-1}$,  is convex, the above Archimedean semi-copula $S_{\overline{G}^{-1}}$ is a copula. Actually $S_{\overline{G}^{-1}}$ is the survival copula of the Schur-constant model in the Example~\ref{example:Schurconstant}.
\end{remark}

In the  particular case of an exchangeable bivariate survival function $\overline{F}(x,y)$ with an Archimedean survival copula $\widehat{C}=C_\phi$, one has
\begin{equation}\label{-F-bar-arch-exchang}
\overline{F}(x,y)=\phi^{-1}\left[ \phi(\overline{G}(x))+\phi(\overline{G}
(y))\right],
\end{equation}
where, as usual, $\overline{G}$ denotes the marginal survival function.
 \begin{remark}\label{rem:B-Archimedean}
 When the survival function $\overline{F}(x,y)$ has the form~(\ref{-F-bar-arch-exchang}), the  ageing function $B(u,v)$  is Archimedean as well. More precisely,  (\ref{DefB-tramite-hatC-e-G-bar}) takes the special form
\begin{align}
B(u,v)&=\exp\{-\overline{G}^{-1}\left( \phi^{-1}\left[ \phi(\big(\overline{G}(-\log u)\big) + \phi\big(\overline{G}(-\log
v)\big)\right]\right)\notag
\}
\intertext{namely}
B(u,v)&= S_\varphi(u,v),
 \label{DefB-tramite-hatC-e-G-bar-ARCHIMEDEAN}
\intertext{where}
\varphi(u)&=\phi\big(\overline{G}(-\log u)\big). \label{varphi}
\end{align}
 \end{remark}

The afore-mentioned dependence properties PQD, LTD and PKD for Archimedean copulas have been investigated in~\cite{AveDort04}. (In~\cite{AveDort04}  and in~\cite{MulScar05} also other positive dependence properties for  Archimedean copulas have also been considered.)
In~\cite{BasSpiz05a} such  dependence properties  for copulas have been considered also for semi-copulas which are not necessarily Archimedean.  We notice  that the PKD property has been however considered therein only the Archimedean case.
 In Section~\ref{section:IFRA-PKD}  we will  extend the results obtained in~\cite{BasSpiz05a}  to a larger class of copulas and  bivariate ageing functions.
\section{Relating Dependence to Ageing Properties: the IFR/DFR case}\label{section:sub-super-migrativity-IFR}

As announced at the end of Section~\ref{section:exchangeable}, we continue to analyze the case of exchangeable variables $X_1,...,X_n$, with  two-dimensional marginal survival function $\overline{F}^{(2)}$  given in~(\ref{F-2}). We denote by $\overline{G}$ the marginal survival function and by $\widehat{C}(u,v)$ the survival copula.
 We are now in a position to start with the discussion about the positive/negative migrativity dependence properties PMD/NMD of super/submigrativity (see~(\ref{SuprMigrtv}), (\ref{SubMigrtv})), univariate ageing properties  IFR/DFR of increasing/decreasing failure rate, and the bivariate positive/negative ageing conditions biv-IFR/biv-DFR (see~(\ref{BivIFRforX}), (\ref{BivDFRforX})).
 The following simple result shows the interrelations among such notions, and can be obtained as a consequence of the general results given in Section 5 of~\cite{BasSpiz05a}. We preferred to state this specific result for different reasons. First of all it admits a self-contained  proof, as given below. Furthermore it provides the reader with a paradigmatic scheme of those  results,
and will  help us to explain in the next section the general idea behind them.
\begin{proposition}\label{propBIS:super/sub/migrative}
${}$
\begin{enumerate}
\item[(i)] If $\overline{G}$ is \emph{IFR} and $\widehat{C}$
 is \emph{PMD}, then $\overline{F}^{(2)}$ is Bayesian \emph{biv-IFR}
\item[(ii)] If $\overline{G}$ is \emph{DFR} and $\widehat{C}$
is \emph{NMD}, then $\overline{F}^{(2)}$ is Bayesian \emph{biv-DFR}
\item[(iii)] If $\overline{F}^{(2)}$ is Bayesian \emph{biv-IFR}  and $\overline{G}$ is \emph{DFR}, then $\widehat{C}$
is \emph{PMD} 
\item[(iv)] If $\overline{F}^{(2)}$ is Bayesian \emph{biv-DFR}  and $\overline{G}$ is \emph{IFR}, then $\widehat{C}$
is  \emph{NMD} 
\item[(v)] If $\widehat{C}$
is  \emph{PMD} 
and $\overline{F}^{(2)}$ is Bayesian \emph{biv-DFR}, then $\overline{G}$ is \emph{DFR}
\item[(vi)] If $\widehat{C}$
is \emph{NMD} 
and $\overline{F}^{(2)}$ is Bayesian \emph{biv-IFR}, then $\overline{G}$
is \emph{IFR}

\end{enumerate}
\end{proposition}

\begin{proof}
It is sufficient to prove the implication in item $(i)$. The other implications
can be proven similarly.
\\

Fix $0\leq x<y,t>0$, so that $\overline{G}\left( x\right) >$ $\overline
{G}\left( y\right) $. Set, furthermore
\[
s:=\frac{\overline{G}\left( x+t\right) }{\overline{G}\left( x\right) },
\]
so that
\[
\frac{\overline{G}\left( y+t\right) }{\overline{G}\left( x\right) }\leq
s\text{ for }\overline{G}\; \text{ IFR}.
\]
By reminding (\ref{Fbar-con-C-hat}) and that $\widehat{C}$ is  supermigrative  we get %
\begin{align*}
\mathbb{P}\left( X_{i}>x+t|X_{i}>x,X_{j}>y\right)
&=\frac{\overline
{F}^{\left( 2\right) }(x+t,y)}{\overline{F}^{\left( 2\right) }(x,y)}%
\\=\frac{
\widehat{C}\left( \overline{G}(x+t),\overline{G}(y)\right) }{\overline
{F}^{\left( 2\right) }(x,y)}
&=\frac{
\widehat{C}\left( \overline{G}(x)\cdot s,\overline{G}(y)\right) }{\overline
{F}^{\left( 2\right) }(x,y)}\geq
\frac{
\widehat{C}\left( \overline{G}(x),\overline
{G}(y)\cdot s\right) }{\overline{F}^{\left( 2\right) }(x,y)} \phantom{\frac{\bigg|}{\bigg|}}
\intertext{By using the assumption of $\overline{G}$ IFR, we have ${G}(y)\cdot s\geq \overline{G}(y+t)$.  Furthermore, $\widehat{C}$~being a copula   and then non-decreasing   in  each variable,  we can conclude}
\mathbb{P}\left( X_{i}>x+t|X_{i}>x,X_{j}>y\right)&\geq
\frac{
\widehat{C}\left( \overline{G}(x),\overline
{G}(y)\cdot s\right) }{\overline{F}^{\left( 2\right) }(x,y)}
\\\geq\frac{\widehat{C}\left( \overline{G}(x),\overline{G}(y+t)\right) }{\overline
{F}^{\left( 2\right) }(x,y)}\phantom{\frac{\bigg|}{\bigg|}}
&=\mathbb{P}\left( X_{j}>y+t|X_{i}>x,X_{j}>y\right) .
\end{align*}
\end{proof}
\bigskip

 In Proposition~\ref{propBIS:super/sub/migrative}, for a joint exchangeable survival function $\overline{F}_{\boldsymbol{X}}$ attention has been fixed  on the two-dimensional marginal survival function $\overline{F}^{(2)}$  and on the pair $(\widehat{C},\overline{G})$ of the two-dimensional survival copula and   the marginal survival function, so that the  $\overline{F}^{(2)}$ is given as in~\eqref{Fbar-con-C-hat-exch}.
 \\
  It is convenient however  to rephrase that result in terms of the triple~$(\widehat{C},\overline{G}, B)$, where $B$ is the bivariate ageing function of $\overline{F}^{(2)}$ (see (\ref{def:B})). On this purpose we recall a result connecting the Bayesian biv-IFR property, with the PMD property $B$ (see~\cite{BasSpiz05a}). For the readers' convenience, we  provide an autonomous proof adapted to our language and notation.

\begin{lemma}\label{lemma-Equiv-SCHUR} The following three conditions are equivalent:
\\
(a) $\overline{F}^{(2)}$ is \emph{Bayesian biv-IFR}, i.e., condition (\ref{BivIFRforX}) holds
\\
(b) $\overline{F}^{\left(
2\right)}$  is Schur-concave
\\
(c) $B$   is \emph{PMD} 
\\
Analogously also the following dual conditions are equivalent
\\
(a') $\overline{F}^{(2)}$ is \emph{Bayesian biv-DFR}, i.e., condition (\ref{BivDFRforX}) holds
\\
(b') $\overline{F}^{\left(
2\right)}$  is Schur-convex
\\
(c') $B$ is \emph{NMD} 
\\
\end{lemma}

\begin{proof}
For $0\leq x<y$, the inequality (\ref{BivIFRforX}) just means
\[
\overline{F}^{\left( 2\right) }(x+t,y)\geq\overline{F}^{\left( 2\right)
}(x,y+t), \quad \text{for any $t \geq 0$,}
\]
namely, in view of exchangeability, the condition~(\ref{def-SCHUR-conc-convex}) of Schur-concavity  for $\overline{F}^{\left( 2\right)}$, i.e.,    condition~$(a)$ is equivalent to condition~$(b)$. It only remains to prove that condition~$(b)$ is equivalent to condition~$(c)$. Indeed, taking into account~(\ref{F-2--Gbar-lnB}), and the fact that the function $\xi \mapsto\overline{G}(-\log( \xi))$ is  increasing, the above Schur-concavity condition is equivalent to
$$
B(e^{-(x+t)}, e^{-y}) \geq B(e^{-x}, e^{-(y+t)}),\quad \text{for any $t \geq 0$, and $0\leq x<y$.}
$$
Setting $u=e^{-x}$, $v=e^{-y}$ and $s=e^{-t}$, the above condition is equivalent to  $B$ being PMD , i.e., to the Supermigrativity condition~(\ref{Super-Sub-migrat-B}) for~$B$.
\\
The proof of the equivalence of the other conditions is similar.
\end{proof}
${}$
\\
\begin{remark}\label{rem:scur-Bayes-biv-closed-mixture}
Notice that the properties (a), (b), (a'), (b')  defined in the above Lemma~\ref{lemma-Equiv-SCHUR} are closed under mixture.
\end{remark}

It is important to stress that the previous result shows that   the property of Supermigrativity (Submigrativity) for the ageing function~$B$ is a bivariate ageing property. 
\\
 In view of the previous equivalences, we are now in a position to rephrase the result in Proposition~\ref{propBIS:super/sub/migrative}
 in terms of the survival copula $\widehat{C}$, the marginal survival function~$\overline{G}$ and the bivariate ageing function~$B$, as follows.
\begin{proposition}\label{cor:super/sub/migrative}${}$
\begin{enumerate}
\item[(i)] If $\overline{G}$ is \emph{IFR} and $\widehat{C}$ is \emph{PMD},
  then~$B$ is \emph{PMD} 
\item[(ii)] If $\overline{G}$ is \emph{DFR} and $\widehat{C}$
is \emph{NMD}, 
then~$B$ is  \emph{NMD} 
\item[(iii)] If $B$ is \emph{PMD} 
and $\overline{G}$ is \emph{DFR}, then $\widehat{C}$
is \emph{PMD} 
\item[(iv)] If $B$ is \emph{NMD} 
 and $\overline{G}$ is \emph{IFR}, then $\widehat{C}$
is \emph{NMD} 
\item[(v)] If $\widehat{C}$
is \emph{PMD} 
and $B$ is \emph{NMD}, 
then $\overline{G}$ is \emph{DFR}
\item[(vi)] If $\widehat{C}$
is \emph{NMD} 
and $B$ is \emph{PMD}, 
then $\overline{G}$
is \emph{IFR}
\end{enumerate}
\end{proposition}

\begin{remark}\label{rem:5.4} In~\cite{AveDort04} it was shown that an Archimedean copula, with a generator~$\phi$ is  \emph{LTI/LTD} if and only if $\phi^{-1}$  is a \emph{IFR/DFR} survival function, respectively. As pointed out in the previous Section~\ref{section:richiami}, it makes sense however, to extend dependence  properties   to semi-copulas. It was observed  in~\cite{BasSpiz05a} that also for an   Archimedean semi-copula  the \emph{LTI/LTD} property  is equivalent to  the \emph{IFR/DFR} property for the inverse of its generator.
In particular $\overline{G}$ is \emph{IFR/DFR} if and only if  the semi-copula $S_{\overline{G}^{-1}}$ (see~(\ref{def:copula-Schur-const}) ) is  \emph{LTI/LTD}. Moreover, again for Archimedean semi-copulas, the \emph{LTI/LTD} property is equivalent to the \emph{NMD/PMD} property (see Remark~\ref{rem:LTD--SuperMigr-ARCH}). We can conclude therefore that the  above result in Proposition~\ref{cor:super/sub/migrative} can   be formulated in terms of the survival copula~$\widehat{C}$, and the semi-copulas $B$ and  $S_{\overline{G}^{-1}}$. For instance item \textit{(i)} can be reformulated as
\begin{equation}
\text{\textit{(i)} \;\; If $S_{\overline{G}^{-1}}$ is \emph{PMD} and $\widehat{C}$ is \emph{PMD},
  then~$B$ is \emph{PMD}. }\qquad \qquad \qquad \text{${}$}
\end{equation}
and similarly for items \textit{(ii)}---\textit{(vi)}.\\

Let us concentrate again our attention on the three semi-copulas~$\widehat{C}$, $B$, and $S_{\overline{G}^{-1}}$:  Summarizing the above arguments, we can claim that   \emph{PMD/NMD} conditions imposed over two semi-copulas imply a condition --- of type either \emph{PMD} or \emph{NMD} --- for the third semi-copula.
\end{remark}

\section{A path to a more general analysis}\label{sec:path}

 As mentioned  in the previous sections, the Supermigrativity  property  for a copula is looked at as a condition of positive
dependence, while the IFR  property of a marginal survival function is a condition of positive one-dimensional
ageing. The terms Submigrativity, DFR respectively refer to the corresponding
dual conditions of negative dependence and negative one-dimensional ageing.
\\
\indent
One can say that Proposition~\ref{cor:super/sub/migrative} concerns with the pro\-perties Super\-migrativity\-/Sub\-migrativity and IFR/DFR  for $\widehat{C}$ and $\overline{G}$, respectively.  It concerns furthermore with   Supermigrativity/Submigrativity of $B$, which has been classified as a  notion of positive/negative bivariate ageing.

Several concepts of stochastic dependence and of ageing have
been considered in the literature. In some cases, results of the same form as
in Proposition~\ref{cor:super/sub/migrative}
can be formulated for other pairs of \lcaporali dependence, one-dimensional ageing\rcaporali properties and notions of \textquotedblleft bivariate ageing",
appropriately corresponding to them.

Informally, we can describe as follows the general format of such results:

\begin{enumerate}
\item[1.] Positive dependence and positive one-dimensional ageing imply positive
bivariate ageing

\item [2.] Positive bivariate ageing and negative one-dimensional ageing imply
positive dependence

\item[3.] Positive bivariate ageing and negative dependence imply positive
one-dimensional ageing.
\end{enumerate}

  Furthermore, dual statements --- where the terms
\textquotedblleft positive" and \textquotedblleft negative" are interchanged --- do hold.

By using the concept of the ageing function $B$, we can say that
\textquotedblleft dependence" is a property of the survival copula $\widehat{C}$, 
\textquotedblleft one-dimensional ageing" is a property of the marginal
survival function $\overline{G}$, and, furthermore \textquotedblleft bivariate ageing" is a property of the ageing function $B$.
See~\cite{BasSpiz05a} for both technical details and some heuristic argument, and~\cite{DurSem05-K}, \cite{DurSem05-I}, \cite{NapSpiz09}, \cite{DurSpiz10}, \cite{Spiz10}, \cite{Spiz14}, \cite{DurSemBook2016}  for further related discussions and results.
\bigskip

In order to setting the above statements in a more precise form, we need to give appropriate answers to the following natural questions:
\\
\indent \textbf{Q1} What are the appropriate
notion of dependence and  property of
one-dimensional ageing, \textquotedblleft corresponding" to each other?
\\
\indent \textbf{Q2} What might be an appropriate notion of bivariate
ageing, \textquotedblleft corresponding" to a pair \lcaporali dependence, one-dimensional
ageing\rcaporali?
\\
\indent  \textbf{Q3} What appropriate meaning should be given to the term \textquotedblleft
corresponding"?
\\
\\
\indent  Finally, assuming that one has given satisfying answers to
the preceding questions, one wonder
\\ -- whether there exists  a general and synthetic method
for proving statements of the same form as in Proposition~\ref{cor:super/sub/migrative}.
In such a case one might
avoid  \textit{ad hoc} strategies,  limited to  specific pairs of \lcaporali dependence, one-dimensional
ageing\rcaporali;
\\ --  whether the notions of bivariate ageing that we are considering are closed under mixtures (see also Remark~\ref{rem:scur-Bayes-biv-closed-mixture}).

\bigskip

Next, we aim to summarize a conceptual path which, for some cases at least,
 suggests   responses to the above questions (see~\cite{BasSpiz05a}). The path starts with the
following three steps a), b), and c).

\bigskip
a)  As also pointed out in Remark~\ref{rem:IFR-univ},   Proposition~\ref{prop:IFR-equiv} ensures that the IFR/DFR
conditions of one-dimensional ageing for $\overline{G}$ can be translated into
inequalities for the bivariate survival function $\overline{F}_\Pi\left(
x,y\right) =\overline{G}\left( x\right) \overline{G}\left(
y\right)$
 of two independent variables distributed according to
$\overline{G}$ (see \emph{(iv)} of Proposition~\ref{prop:IFR-equiv}, see also   \emph{(iii)}, which turns out to be an inequality in view of~(\ref{def-SCHUR-conc-convex}) ).
We can consider, in place of IFR/DFR conditions,  the NBU/NWU conditions for $\overline{G}$, i.e.,
$$
\overline{G}(x+y) \leq (\geq) \overline{G}(x)\overline{G}(y),$$
 or equivalently
\begin{equation}\label{NBU-level}
 \overline{F}_\Pi(x+y,0)\leq (\geq) \overline{F}_\Pi(x,y).
\end{equation}

\bigskip

b) As observed in Remark~\ref{rem:IFR-univ}   the inequalities in the above step  a), concerning the IFR/DFR properties,  can be seen as properties
of the family of the level curves of the function $\overline{F}_\Pi\left(
x,y\right) =\overline{G}\left( x\right) \overline{G}\left(
y\right)$. Also the NBU/NWU properties~(\ref{NBU-level})  is a property of the family of the level curves for $\overline{F}_\Pi$.
Any
exchangeable bivariate survival function $\overline{F}(x,y)$,
which shares with $\overline{F}_\Pi(x,y)$ the family of level curves,
also shares with $\overline{F}_\Pi$ qualitative properties of
bivariate ageing, as first discussed in~\cite{BarSpiz93}, \cite{BasSpiz01},
\cite{Spiz01}.
On the other hand,
we recall that two different survival functions   share the same family of level curves if and only if they share the same function ageing $B$. Whence,  different bivariate ageing properties defined in terms of the family of the level curves of $\overline{F}(x,y)$  turn out to coincide with
\textquotedblleft dependence"
properties of $B$.

\bigskip

c) It is remarkable  that the bivariate ageing properties, mentioned in item~b) and defined in terms of
properties of the ageing function $B$,
actually  coincide with different
\textquotedblleft dependence" properties of $B$ (see~\cite{BasSpiz05a}, see also Remark~\ref{rem:5.4}).
The special case of PMD/NMD is   considered here in Lemma~\ref{lemma-Equiv-SCHUR}.
 The special case  when $B$ is the independent copula $\Pi$ corresponds to a bivariate no-ageing property.

\bigskip

Before continuing our conceptual path, we point out that different papers in the literature
had been devoted to the analysis of  connections between dependence properties of $C_{\phi}$ and ageing properties
of $\phi^{-1}$. Very precise and detailed results had, in particular, been
obtained in~\cite{AveDort04}, \cite{MulScar05}. Some notions of positive
dependence for $C_{\phi}$ have been shown to be equivalent to negative ageing
properties of $\phi^{-1}$.
\\
For the reader's ease, we recall the results in~\cite{AveDort04} concerning the relations between the dependence properties PQD,  PKD, LTD, SI for an Archimedean copula $C_\phi$, and the ageing properties of the survival functions $\phi^{-1}$. For simplicity sake we assume that the generator $\phi$ is differentiable in $(0,1)$ and denote by $f_\phi(x)$ the density of~$\phi^{-1}$.
\begin{align}
C_\phi\text{ is PQD} \quad &\Leftrightarrow \quad  \phi^{-1}(x)\; \text{is NWU;}\label{A-D-PQD}
\\
C_\phi\text{ is PKD} \quad &\Leftrightarrow \quad  \phi^{-1}(x)\; \text{is DFRA;}\label{A-D-PKD}
\\
C_\phi\text{ is LTD} \quad &\Leftrightarrow \quad  \phi^{-1}(x)\; \text{is DFR;}\label{A-D-LTD}
\\
C_\phi\text{ is SI}\phantom{D\,} \quad &\Leftrightarrow \quad  \log (f_\phi)\; \text{ is convex.}\label{A-D-SI}
\end{align}
For the DFRA property, see Definition~\ref{def:IFRA-DFRA} in the next section.
\\

Similar equivalences hold for the corresponding properties of negative dependence and positive ageing.
\\
We recall that  an Archimedean copula is  LTD  if and only if is PMD (see also Remark~\ref{rem:5.4}).
We also notice the following implications between notions of positive univariate negative ageing (see, e.g., \cite{ShaSha2007})
$$
\log (f_\phi)\; \text{convex}\quad \Rightarrow\quad \text{DFR} \quad \Rightarrow\quad \text{DFRA} \quad \Rightarrow\quad \text{NWU}.
$$
Whence one obtains the following implications  among notions of bivariate dependence for Archimedean copulas
$$
\text{SI}  \quad \Rightarrow\quad  \text{LTD} \quad \Rightarrow\quad  \text{PKD}\quad \Rightarrow\quad \text{PQD}.
$$

\bigskip

We can now    resume our conceptual path by means of the following steps e)---h).
\\

e) We associate to the Archimedean copula $C_\phi$ the bivariate survival model in $\mathbb{R}^+\times \mathbb{R}^+$
\begin{equation}\label{Schur-Constant-phi}
\overline{F}(x,y)= C_\phi\big(\phi^{-1}(x), \phi^{-1}(y) \big),
\end{equation}
i.e., the one with survival copula $\widehat{C}= C_\phi$ and marginal survival function~\mbox{$\overline{G}=\phi^{-1}$.} In other words this is the Schur-constant model (see~(\ref{-F-bar-arch-exchang}))
$$
\overline{F}(x,y)=  \overline{G}(x+y).
$$

 The interest of this model is due to the circumstance that  the equivalences~(\ref{A-D-PQD})---(\ref{A-D-SI}) become nothing else but equivalence relations between properties of the survival copula and ones  of the marginal distribution.

Notice that the ageing function $B$ of the above model~(\ref{Schur-Constant-phi}) is  the independent copula $\Pi$  and, in agreement with item c), the condition $B=\Pi$ actually describes a bivariate no-ageing property.

\bigskip

f) We now   consider a survival model, with the same  survival copula $\widehat{C}= C_\phi$ as in item e), but with marginal survival function~\mbox{$\overline{G}$}, different from $\phi^{-1}$. In this case the ageing function $B$ is different  from $\Pi$. As pointed out in Section~\ref{section:richiami},   $B$ is not necessarily a copula, however it turns out to be the Archimedean  semi-copula  $S_\varphi$ in~(\ref{DefB-tramite-hatC-e-G-bar-ARCHIMEDEAN}) with generator $\varphi$ given in~(\ref{varphi}), i.e.,
$$\varphi(u)=\phi\big(\overline{G}(-\log(u))\big).$$
Actually $B=S_{\varphi}$ is  a copula if and only if $\varphi$ is convex. Even if this is not the case, one is still allowed to consider \textquotedblleft
extended" dependence properties of $B=S_\varphi$, namely bivariate ageing properties of the survival model.
 In this
perspective, the results obtained in~\cite{AveDort04}, \cite{MulScar05} can be
extended to Archimedean  semi-copulas and reformulated as equivalence relations between dependence properties of $B=S_{\varphi}$ and univariate ageing properties of the generator's inverse  $\varphi^{-1}$.

\bigskip

g) We remind that one purpose of our discussion is the analysis of the interrelations among stochastic dependence of $\overline{F}=\overline{F}^{(2)}$, and ageing properties of the marginal survival function $\overline{G}$.
In the Archimedean case of item f), stochastic dependence of $\overline{F}(x,y)=\overline{F}^{(2)}(x,y)=C_\phi(\overline{G}(x), \overline{G}(y))$
can be characterized by ageing property of survival function $\phi^{-1}$,  in view of the equivalences~(\ref{A-D-PQD})---(\ref{A-D-SI}) above. By the same token, bivariate ageing can be characterized in terms of the univariate ageing property of survival function~$\varphi^{-1}$. By imposing one fixed condition of ageing on each of the three survival functions, $\phi^{-1},\, \varphi^{-1}$, $\overline{G}$,  we respectively  obtain a property of dependence, of bivariate ageing, and of univariate ageing for the survival model $\overline{F}=\overline{F}^{(2)}$.
These survival functions are linked together by the relation~(\ref{varphi}), which we rewrite in the form
\begin{equation}
\overline{G}(x)=\phi^{-1}\big( \varphi(e^{-x})\big),
\end{equation}
whence, imposing  the fixed  property of ageing on two of these survival functions implies a condition for the third one. We give an example of this procedure in Lemma~\ref{lemma:G=H1-H2}.
\\
\indent By taking again into account the equivalence properties of \mbox{type~(\ref{A-D-PQD})---(\ref{A-D-SI}),} we convert such implications  (for  $\phi^{-1},\, \varphi^{-1}$, $\overline{G}$)
into implications (for  $C_\phi,\, S_\varphi$,  $\overline{G}$) of the form appearing in items 1., 2., and 3.\@ at the beginning of this section. Finally, observing that the univariate ageing properties of $\overline{G}$  can be characterized by the dependence properties of the semi-copula $S_{\overline{G}^{-1}}$, the previous implications can be rephrased as implications on the semi-copulas $C_\phi,\, S_\varphi$,  $S_{\overline{G}}$ (as already pointed out in Remark~\ref{rem:5.4} for the IFR/DFR property)

\bigskip

The announced conceptual path concludes with the following item.
\bigskip

h) From the arguments in the above items, and in particular item g), we see that, for the Archimedean case, it is equivalent to state results in terms of stochastic dependence of semi-copulas or in terms of univariate ageing of the inverse of their generators. However the description in term of stochastic dependence allows for a more general analysis. Actually, we can deal with stochastic dependence of semi-copulas, even for non-Archimedean semi-copulas.
\\
\indent We highlight that generally properties of dependence can be characterized in terms of appropriate comparisons with the independent case, i.e.,    between the two models
 $$
\overline{F}(x,y)= \widehat{C}\big(\overline{G}(x),\overline{G}(y) \big)\quad \text{and}  \quad   \overline{F}_\Pi(x,y)=\Pi(\overline{G}(x),\overline{G}(y)).
$$
This comparison becomes a comparison between $\widehat{C}$ and $\Pi$.
 For instance the PQD property means that
$$
\overline{F}(x,y) \geq \overline{F}_\Pi(x,y)=\overline{G}(x)\overline{G}(y),
$$
or equivalently
$$
\widehat{C}(u,v)\geq \Pi(u,v)=u\,v.
$$
Similarly, we consider properties of bivariate ageing which can be characterized in terms of appropriate comparisons with the Schur-constant case, i.e., between the two functions
$$
\overline{F}(x,y)
\quad  \text{and} \quad   \overline{G}(x+y).
$$
We observe that the above functions can be written as
$$
\overline{F}(x,y)=\overline{G}\big(-\log\big(B(e^{-x},e^{-y})\big)\big)
\quad  \text{and} \quad   \overline{G}(x+y)= \overline{G}\big(-\log\big(\Pi(e^{-x},e^{-y})\big)  \big),
$$
and therefore this comparison becomes now a comparison between $B$ and $\Pi$.
For instance the bivariate NBU  property is defined by
$$
\overline{F}(x,y) \leq  \overline{G}(x+y),
$$
which can also be written
$$
B(u,v) \geq \Pi(u,v)=u\,v.
$$

This approach  justifies the fact that   bivariate ageing properties of $\overline{F}(x,y)$ are defined in terms of dependence properties of $B$, as done in~\cite{BasSpiz05a}. More precisely we now give a partial answer to the above questions \textbf{Q1}--\textbf{Q3}, by the following claim:
\\
\emph{
Let $\preceq$ be a Positive   Dependence Ordering (see, e.g., \cite{KimSam1989}), and consider the Positive Dependence Property for the survival copula defined by the condition
\begin{equation}\label{C-POD-Pi}
\widehat{C}(u,v) \succeq \Pi(u,v).
\end{equation}
Then the  \textquotedblleft corresponding"   Positive univariate Ageing Property  for $\overline{G}$ and
Positive bivariate   Ageing  Property  for $\overline{F}(x,y)$ are
\begin{equation}\label{G-B-POD-Pi}
S_{G^{-1}}(u,v) \succeq \Pi(u,v)\quad \text{and}\quad  B(u,v) \succeq \Pi(u,v).
\end{equation}
}

\bigskip

\begin{remark}
When $B$ is a copula, then $\overline{F}_B(x,y)=B(e^{-x},e^{-y})$ is a true bivariate survival function, with standard exponential marginals, as well as $\Pi(e^{-x},e^{-y})$. Furthermore, since $B$ does  clearly coincide with the survival copula of $\overline{F}_B(x,y)$, and simultaneously is also the bivariate ageing of it, the dependence property of such a model coincides with the bivariate ageing property. The latter is also the ageing property of  $\overline{F}$, since $\overline{F}$ and $\overline{F}_B$  share the same family of level curves.
\end{remark}
\bigskip

The path described so far  unifies the treatment of the Archimedean and non-Archimedean models.  It remains to check that items 1., 2., and 3.\@ at the beginning of this section hold true for the above \textquotedblleft corresponding" ageing/dependence properties.
 Results of this type have   been obtained in~\cite{BasSpiz05a} concerning with some specific ageing/dependence properties. In that paper the case dealing with the PKD property was an exception, in the sense that only Archimedean models were considered. The concept of generalized Kendall distribution and  related equivalence classes of semi-copulas (see Definition~\ref{def:generalized-Kendall-class} below) allow us to treat the PKD property for a class of non-Archimedean models, as shown in the next Section~\ref{section:IFRA-PKD}.

\section{Relating Dependence to Ageing Properties: the IFRA/DFRA case}\label{section:IFRA-PKD}

In the Reliability literature (see~\cite{BirEsaMar1966}, \cite{BarPros1975}, \cite{LaiXie2006}, \cite{ShaSha2007}) relevant concepts of positive/negative one-dimensional ageing are those of Increasing/Decreasing Failure Rate in Average (IFRA/DFRA), which, respectively, generalize those of IFR/DFR and are defined as follows.
\begin{definition}\label{def:IFRA-DFRA}
A one dimensional survival function $\overline{G}$ is \emph{IFRA}  if  and only
if
$$
x \mapsto -\frac{\log \overline{G}(x)}{x}
$$
is an increasing function
(see e.g.\@   \cite{BarPros1975});
 $\overline{G}$ is \emph{DFRA} if and only
if it
 is a non-decreasing function.
\end{definition}

This notion of ageing is strictly related with the notion of PKD/NKD.
As shown in~\cite{AveDort04}, in fact, an Archimedean copula (with a differentiable generator $\phi$) is  NKD/PKD, i.e., $\phi(u)/\phi^\prime(u)\leq u\, \log(u)$ if and only if $\phi^{-1}$  is a IFRA/DFRA survival function, respectively (see~(\ref{A-D-PKD})).

Let us now come to an exchangeable model  with the Archimedean  survival copula $\widehat{C}=C_\phi$ (with a differentiable generator $\phi$), and marginal survival function $\overline{G}$. As we observed in  Remark~\ref{rem:B-Archimedean}, the corresponding
ageing function is $B=S_\varphi$, where we are using the notation given in (\ref{DefB-tramite-hatC-e-G-bar-ARCHIMEDEAN}) and~(\ref{varphi}), i.e., $\varphi(u)= \phi\big(\overline{G}(-\log(u))\big)$. Inspired by the latter circumstance, in~\cite{BasSpiz05a} the PKD/NKD property was extended to Archimedean ageing functions (not necessarily copulas) by requiring $\varphi(u)/\varphi^\prime(u)\leq u\, \log(u)$, and it was pointed out that the above equivalence between NKD/PKD and IFRA/DFRA  holds even for Archimedean ageing functions.

For such an exchangeable model,  a result similar to Proposition~\ref{cor:super/sub/migrative} was given in~\cite{BasSpiz05a} (see Example~7.4 therein), for what concerns relations between PKD/NKD properties for ageing functions and IFRA/DFRA. Such a result
can be rephrased here as follows.

\begin{proposition} \label{prop:PKD-BAS-SPIZ}
${}$\\\vspace{-3mm}
\begin{enumerate}
\item[(i)] If  $\overline{G}$ is \emph{IFRA} and $\widehat{C}=C_\phi$ is \emph{PKD}
then $B=S_\varphi$ is \emph{PKD}
\item[(ii)] If $\overline{G}$ is \emph{DFRA} and $\widehat{C}=C_\phi$ is \emph{NKD}   
then $B=S_\varphi$ is \emph{NKD}
\item[(iii)] If $B=S_\varphi$ is \emph{PKD}   and $\overline{G}$ is \emph{DFRA}    then $\widehat{C}=C_\phi$ is \emph{PKD}
\item[(iv)] If $B=S_\varphi$ is \emph{NKD} and $\overline{G}$ is \emph{IFRA} 
then $\widehat{C}=C_\phi$ is \emph{NKD}
\item[(v)]  If $\widehat{C}=C_\phi$ is \emph{PKD} and $B=S_\varphi$ is \emph{NKD}    then
$\overline{G}$ is \emph{DFRA}
\item[(vi)] If $\widehat{C}=C_\phi$ is \emph{NKD} and $B=S_\varphi$ is \emph{PKD}   then
$\overline{G}$ is \emph{IFRA}

\end{enumerate}

\end{proposition}

\begin{remark}\label{AveDort04-BasSpiz05a}

We remind that  we defined  the generalized Kendall distribution $K_S$ (see Definition~\ref{Def:Kendall-semicopula})
 for any semi-copula $S$ (neither necessarily Archimedean nor ageing function)
 and, consequently, implicitly extended the \emph{PKD} property to semi-copulas, by requiring
 \begin{align}\label{PKD-semicopula}
K_S(t)\geq K_\Pi(t)=t-t\,\log(t),
\end{align}
and similarly for the \emph{NKD} property.
\\
In view of Example~\ref{example:KS_phi}, we can claim that Proposition~\ref{prop:PKD-BAS-SPIZ} also holds with our extended definition of \emph{PKD/NKD}.
 \end{remark}

Here we  show that the previous result admits a natural extension to a larger class of exchangeable models, where the survival copula $\widehat{C}$ is not necessarily Archimedean. On this purpose we start by
extending to semi-copulas the equivalence relation based on the Kendall distribution, introduced in~\cite{NelAl09} for  copulas.
\begin{definition}\label{def:generalized-Kendall-class}
Two bivariate semi-copulas $S_1(u,v)$ and $S_2(u,v)$ are \emph{Kendall-equivalent}, written as $S_1\equiv_K S_2$, if and only if the Kendall  distributions $K_{S_1}(t)$ and $K_{S_2}(t)$ do coincide.
\\  For a fixed semi-copula $S(u,v)$ we denote by $\mathcal{C}_S$ the equivalence class containing~$S$.
\end{definition}

We can thus claim that if a semi-copula $S$ is PKD/NKD then, by definition, all the semi-copulas in the equivalence class  $\mathcal{C}_S$
share  the same Kendall dependence property.
The latter claim suggests us that the appropriate extension of Proposition~\ref{prop:PKD-BAS-SPIZ} is obtained by replacing the conditions $\widehat{C}=C_\phi$ and $B=S_\varphi$ with the conditions $\widehat{C} \in \mathcal{C}_{C_\phi}$ and $B\in \mathcal{C}_{S_\varphi}$, respectively. A precise result will be given in Proposition~\ref{prop:PKD-NAP-SPIZ} below. To this end we also introduce the following definition.

\begin{definition}\label{def:pseudo-Archimedean}
A semi-copula $S$ is \emph{pseudo-Archimedean} whenever there exists an Archimedean semi-copula $S_\varphi$ such that   $S\in \mathcal{C}_{S_\varphi}$. The generator $\varphi$ of $S_\varphi$ will be referred to as the \emph{pseudo-generator of}~$S$.
\end{definition}

In the perspective of extending  Proposition~\ref{prop:PKD-BAS-SPIZ}, we  recall (Proposition~\ref{prop:K-di-C-hat---K-di-B}) how to compute explicitly  the Kendall distribution for a class of copulas larger than the Archimedean one, as considered in~\cite{GenRiv01}.
In this respect we remind that, following what was done for copulas in~\cite{GenRiv01},   attention in~\cite{NapSpiz09} was restricted to models with  the following properties:
\begin{description}
\item{(P1)} for any $u\in (0,1)$ the function $v \mapsto \widehat{C}_u(v):=\widehat{C}(u,v)$ is strictly increasing and continuous (and therefore invertible),
    \item{(P2)} the function $x \mapsto \overline{G}(x)$ is strictly decreasing and continuous (and therefore invertible).
\end{description}

Furthermore (still in Proposition~\ref{prop:K-di-C-hat---K-di-B}) we will analyze the relation between the Kendall distribution of a copula and the generalized one of the corresponding ageing function
$B(u,v)$: To this end we recall that $B(u,v)$ can be written in terms of $ \widehat{C}(u,v)$, $\gamma (u)= \exp\{-\overline{G}^{-1}(u) \}$, and $\gamma^{-1} (z)= \overline{G}(-\log(z))$ as in~(\ref{def:B-con-C-hat-e-gamma}), i.e.,
$$
B(u,v)=\gamma\big(\widehat{C}(\gamma^{-1}(u),\gamma^{-1}(v) \big).
$$

\begin{remark}
When $\overline{G}(x)$ is strictly positive all over $[0,\infty)$, as we have assumed in this paper, then the functions $\gamma(u)$ and $\gamma^{-1}(z)$ are strictly increasing and continuous with $\gamma{(0)}=0$ and $\gamma{(1)}=1$.

It is thus clear that the conditions (P1) and (P2) imply that for any $u\in (0,1)$ the function $v \mapsto B_u(v):=B(u,v)$ is strictly increasing and continuous (and therefore invertible).

In conclusion either the generalized inverse $C^{-1}_u(t)$ or  $B^{-1}_u(t)$  are true inverse functions.
\end{remark}

We are thus in a position to compute the Kendall distributions of $\widehat{C}$ and~$B$, respectively, by using Proposition~1 in~\cite{GenRiv01} for copulas, and its extension to ageing functions.

\begin{proposition}\label{prop:K-di-C-hat---K-di-B}
Assume properties (P1) and (P2). Then, for any $t\in[0,1]$,
\begin{align}
K_{\widehat{C}}(t)&=t+\int_{t}^{1}\frac{\partial\widehat{C}}{\partial u}(u,v)|_{v=\widehat{C}_{u}^{-1}(t)}\,du,\phantom{\frac{.}{ \big|}} \label{mathcal-K-di-C-hat}%
\intertext{Assume furthermore that the density $g$ of $G$ is continuous, then}
K_{B}(t)&=t+ \gamma^\prime (\gamma^{-1}(t))\, \big[K_{\widehat{C}}\big(\gamma^{-1}(t)\big) - \gamma^{-1}(t) \big]\phantom{\frac{.}{ \big|}}\label{KB-in-terms-KC-e-gamma}
\\
&=t+ \frac{1}{\frac{d}{dt}  \overline{G}(-\log(t))} \,\Big[ K_{\widehat{C}}\big(\overline{G}(-\log(t))\big) - \overline{G}(-\log(t))\Big], \phantom{\frac{.}{ \bigg|}}\label{KB-in-terms-KC-e-bar-G}
\\
K_{B}(t)&=t+\int_{t}^{1}\frac{\partial B}{\partial u}(u,v)|_{v=B_{u}^{-1}(t)}\,du. \label{mathcal-K-di-B}
\end{align}

\end{proposition}
\begin{proof} We essentially need to prove~(\ref{KB-in-terms-KC-e-gamma}) for $t\in (0,1)$. In fact
\begin{description}
\item[-]
Equation~(\ref{mathcal-K-di-C-hat}) is exactly Proposition~1 in~\cite{GenRiv01},
\item[-]  Equation~(\ref{KB-in-terms-KC-e-bar-G}) just amounts to a reformulation of~(\ref{KB-in-terms-KC-e-gamma}), by rewriting $\gamma^{-1}(t)$ and $\gamma^\prime (\gamma^{-1}(t))$ in terms of $\overline{G}$.
\item[-] Equation~(\ref{mathcal-K-di-B})  easily follows by~(\ref{mathcal-K-di-C-hat}) and~(\ref{KB-in-terms-KC-e-gamma}) by taking into account that
\begin{align}
K_{B}(t)&=t+ \gamma^\prime (\gamma^{-1}(t))\,\int_{\gamma^{-1}(t)}^1
\frac{\partial\widehat{C}}{\partial u}(u,v)|_{v=\widehat{C}_{u}^{-1}(\gamma^{-1}(t))}\,du,\label{KB-in-terms-KC}
\end{align}
and by the change of variable $u^\prime=\gamma(u)$ within the previous integral.\\
 In this respect we point out that
 $$
 \text{$v=\widehat{C}_{u}^{-1}(\gamma^{-1}(t))$ iff $\gamma(v)=B_{\gamma(u)}^{-1}(t)$, }
 $$
  $$
  \text{$\widehat{C}(u,v)|_{v=\widehat{C}_{u}^{-1}(\gamma^{-1}(t))}=t$,\quad  $1=\gamma^{-1}(1)$.}
  $$
 \item[-]Equation~(\ref{KB-in-terms-KC-e-gamma}) is obvious for $t=0$ and $t=1$.
\end{description}
In view of~(\ref{eq:KS-sup-t+}), the proof of~(\ref{KB-in-terms-KC-e-gamma}) for $t\in (0,1)$ is obtained by proving that
$$
K_B(t)-t=\sup_{\mathcal{P}} \Big\{\sum_{ \underset{u_{i+1}> t}{i\in I}} [B(u_{i+1},v_i)-B(u_i,v_i)], \; \text{with }\; v_i=B^{-1}_{u_{i+1}}(t)\Big\}
$$
coincides with $\gamma^\prime (\gamma^{-1}(t))\, \big[K_{\widehat{C}}\big(\gamma^{-1}(t)\big) - \gamma^{-1}(t) \big]$.
In fact, setting
$$
w_i=\gamma^{-1}(u_i), \qquad z_i= \gamma^{-1}(v_i)
$$
and observing that
$$
z_i= \gamma^{-1}\big(B^{-1}_{u_{i+1}}(t)\big)
=\widehat{C}^{-1}_{w_{i+1}}(\gamma^{-1}(t)),
$$
we can write $K_B(t)-t$ as
\begin{align*}
&\sup\{\sum_{ \underset{w_{i+1}> \gamma^{-1}(t)}{i\in I}}  [\gamma \big(\widehat{C}(w_{i+1},z_i)\big)-\gamma\big(\widehat{C}(w_{i},z_i)\big)],\; z_i=\widehat{C}^{-1}_{w_{i+1}}(\gamma^{-1}(t))\}.
\intertext{
Setting $\Delta_i = \widehat{C}(w_{i+1},z_i)\big)-\widehat{C}(w_{i},z_i)$ we immediately get}
&\gamma \big(\widehat{C}(w_{i+1},z_i)\big)-\gamma\big(\widehat{C}(w_{i},z_i)\big)= \gamma^\prime\big(\widehat{C}(w_{i+1},z_i)\big) \, \Delta_i
+ Err_i,
\intertext{with}
 Err_i&=
\big[\gamma^\prime\big(\gamma^{-1}(t)+ \theta_i \, \Delta_i \big)-\gamma^\prime\big(\gamma^{-1}(t)\big)\big]\, \Delta_i, \quad \text{for a suitable $\theta_i\in (0,1)$.}
\end{align*}
As it is well known, any copula is a Lipschitz function, with Lipschitz constant equal to $1$, therefore $\widehat{C}(u,v)$  being  a copula, we have
$$
0\leq \Delta_i =  \widehat{C}(w_{i+1},z_i)\big)-\widehat{C}(w_{i},z_i)\leq w_{i+1}-w_{i}\leq \delta,
$$
where we have set
$$
\delta:= \max (w_{i+1}-w_{i}; \; t\leq w_i \leq w_{i+1} \leq 1).
$$
Without loss of generality we may assume that
  $t+\delta  < 1 $,  so that
$$
Err_i\leq \sup_{s \in [t, t+\delta]} |\gamma^\prime\big(\gamma^{-1}(s)\big)-\gamma^\prime\big(\gamma^{-1}(t)\big)|\, (w_{i+1}-w_i).
$$
Finally, observing that $\widehat{C}(w_{i},z_i)=\gamma^{-1}(t)$, we have that
\begin{align*}
&|\sum_{ \underset{w_{i+1}> \gamma^{-1}(t)}{i\in I}}  [\gamma \big(\widehat{C}(w_{i+1},z_i)\big)-\gamma\big(\widehat{C}(w_{i},z_i)\big)]-
\sum_{ \underset{w_{i+1}> \gamma^{-1}(t)}{i\in I}} \gamma^\prime\big(\gamma^{-1}(t)\big) \, \Delta_i |
\\
\leq &\sum_{ \underset{w_{i+1}> \gamma^{-1}(t)}{i\in I}}\sup_{s \in [t, t+\delta]} |\gamma^\prime\big(\gamma^{-1}(s)\big)-\gamma^\prime\big(\gamma^{-1}(t)\big)|\, (w_{i+1}-w_i)
\\&= \sup_{s \in [t, t+\delta]} |\gamma^\prime\big(\gamma^{-1}(s)\big)-\gamma^\prime\big(\gamma^{-1}(t)\big)|
\end{align*}
Recalling that  $g(x)$ denotes the continuous density of $G$, we rewrite   $\gamma^\prime\big(\gamma^{-1}(t)\big)= \frac{t}{g(-\log(t))}$. By assumption the latter function is  continuous  in $(0,1)$,
and the proof of~(\ref{KB-in-terms-KC-e-gamma}) is accomplished by observing that
$$
\sup\{\sum_{ \underset{w_{i+1}> \gamma^{-1}(t)}{i\in I}}  \Delta_i, \; \text{with}\;z_i=C^{-1}_{w_{i+1}}(\gamma^{-1}(t))\}=
K_{\widehat{C}}\big(\gamma^{-1}(t)\big) - \gamma^{-1}(t) .
$$
\end{proof}

\begin{remark}
Apparently,
the equation~(\ref{KB-in-terms-KC-e-gamma})  is nothing else but an equivalent form to write the thesis of Proposition~20 in~\cite{NapSpiz09}. Actually the difference between the two results  lies  in the present generalization, to semi-copulas, of Kendall distributions.
In~\cite{NapSpiz09}   such a generalization was not considered, rather   the right hand side of~(\ref{mathcal-K-di-B}) was taken as an operator on ageing functions.  Here, on the contrary, $K_B(t)$ is defined directly by~(\ref{eq:KS-sup}), or, equivalently by~(\ref{eq:KS-sup-t+}), and equation~(\ref{mathcal-K-di-B}) is obtained as a direct consequence.
\end{remark}

The present approach allows us to highlight that $B$ is pseudo-Archimedean if and only if $\widehat{C}$ is such, and to obtain the pseudo-generators of them. Finally,  as a crucial issue, we can thus obtain the relation tying $\overline{G}$ with such pseudo-generators. The precise statements follow:
\begin{proposition}\label{prop:G=phi-varphi-1}
Assume conditions (P1) and (P2).
Fix a $t_0\in (0,1)$ and define
\begin{equation}\label{eq:phi-t-con-K-C-hat}
\phi(t):=\exp\left\{ \int_{t_0}^t \frac{1}{s-K_{\widehat{C}}(s)}\, ds\right\},
\end{equation}
and
\begin{equation}\label{eq:varphi-t-con-K-C-hat}
\varphi(t):=\exp\left\{ \int_{\gamma(t_0)}^t \frac{1}{s-K_{B}(s)}\, ds\right\}.
\end{equation}
Then
\\
(i) $\phi$ is the generator of an Archimedean copula, namely $C_\phi$, and $\varphi$ is the generator of an Archimedean semi-copula, $S_\varphi$,
\\(ii) $\widehat{C}$ and $B$ are pseudo-Archimedean and, in particular,
 $\widehat{C} \in \mathcal{C}_{C_\phi}$
 and $B\in \mathcal{C}_{S_\varphi}$.
 \\
Assume furthermore that the density $g$ of $G$ is continuous, then
\vspace{2mm}\\
(iii) \qquad \qquad \qquad \quad   \quad   $\varphi(t)= \phi(\gamma^{-1}(t))= \phi\big(\overline{G}(-\log(t)) \big)$,
 \\
or, equivalently,
$$\overline{G}(x)= \phi^{-1}( \varphi (e^{-x})).$$
\end{proposition}
\begin{proof}
By Proposition 1.2 in~\cite{GenRiv1993}, and as observed in~\cite{NelAl09}, we know that, for any ''true'' Kendall distribution function $K(t)$ of a bivariate survival function, the function
$\phi_K(t):=\exp\left\{ \int_{t_0}^t \frac{1}{s-K(s)}\, ds\right\}$ is the generator of an Archimedean copula, i.e., $\phi_K(t)$ is increasing and convex, with  $\phi_K(0)=1$, if and only if $K(t^-) > t $ for all $t\in (0,1)$. Moreover $K(t)$ is the Kendall distribution of the  Archimedean copula $C_{\phi_K}$. \\
Under (P1) and (P2), by~(\ref{mathcal-K-di-C-hat}) in the previous Proposition~\ref{prop:K-di-C-hat---K-di-B}, the Kendall distribution function of $\widehat{C}$ satisfies $K_{\widehat{C}}(t^-) >t$ for all $t\in (0,1)$, so that $\phi(t)$ is the generator of $C_\phi$, and clearly
$$
K_{\widehat{C}}(t)= K_{C_\phi}(t)=  t-\frac{\phi(t)}{\phi^\prime(t)},
$$
and \emph{(i)-(ii)} follow as far as $\phi(t)$ and $\widehat{C}$ are concerned.
\\
 Similar results hold
 except for the convexity property (see also~\cite{NapSpiz09}), as far as $\varphi(t)$ and $B$ are  concerned, in particular
 $$
 K_B(t)= K_{S_\varphi}(t)=t-\frac{\varphi(t)}{\varphi^\prime(t)}.
 $$
\indent Finally, by using the above expressions of $ K_{\widehat{C}}(t)$ and $K_B(t)$, together with~(\ref{KB-in-terms-KC-e-gamma}), one immediately gets that $\varphi(t)= \phi(\gamma^{-1}(t))$, i.e., \emph{(iii)} follows.
\end{proof}

The announced generalization of~Proposition~\ref{prop:PKD-BAS-SPIZ} is now proved by collecting the results in Propositions~\ref{prop:G=phi-varphi-1}, Remark~\ref{AveDort04-BasSpiz05a}, and Proposition~\ref{prop:PKD-BAS-SPIZ} itself.

\begin{proposition}\label{prop:PKD-NAP-SPIZ}
Let $F(x,y)$ be a bivariate exchangeable model satisfying the conditions (P1) and (P2). Besides the standing assumptions on $\overline{G}$, assume that the corresponding density $g(x)$ is continuous. Finally assume that the generator~$\phi(t)$ defined in~(\ref{eq:phi-t-con-K-C-hat}) is differentiable. Then
\begin{enumerate}
\item[(i)] If  $\overline{G}$ is \emph{IFRA} and $\widehat{C}$ is \emph{PKD}
then $B$ is \emph{PKD}
\item[(ii)] If $\overline{G}$ is \emph{DFRA} and $\widehat{C}$ is \emph{NKD}   
then $B$ is \emph{NKD}
\item[(iii)] If $B$ is \emph{PKD}   and $\overline{G}$ is \emph{DFRA}    then $\widehat{C}$ is \emph{PKD}
\item[(iv)] If $B$ is \emph{NKD} and $\overline{G}$ is \emph{IFRA} 
then $\widehat{C}$ is \emph{NKD}
\item[(v)]  If $\widehat{C}$ is \emph{PKD} and $B$ is \emph{NKD}    then
$\overline{G}$ is \emph{DFRA}
\item[(vi)] If $\widehat{C}$ is \emph{NKD} and $B$ is \emph{PKD}   then
$\overline{G}$ is \emph{IFRA}
\end{enumerate}
\end{proposition}

In the proof of the above Proposition~\ref{prop:PKD-NAP-SPIZ} we have used the result of~\cite{BasSpiz05a}, and recalled in Proposition~\ref{prop:PKD-BAS-SPIZ}.  In order to get a self-contained proof, instead of Proposition~\ref{prop:PKD-BAS-SPIZ}, we can use the following Lemma~\ref{lemma:G=H1-H2} concerning the IFRA/DFRA property of three survival functions $\overline{G}$, $\overline{H}_1$, and~$\overline{H_2}$. Indeed,  in the spirit of item g) in the previous Section~\ref{sec:path}, and taking into account that the generalized Kendall distribution of pseudo-Archimedean semi-copulas coincides with the Kendall distribution of true Archimedean semi-copulas, we only need  to apply Lemma~\ref{lemma:G=H1-H2} to $\overline{H}_1=\phi^{-1}$ and $\overline{H}_2=\varphi^{-1}$.
\begin{lemma}\label{lemma:G=H1-H2}
Let  $\overline{H}_1(x)$ and $\overline{H}_2(x)$ be two continuous survival functions
  strictly decreasing,  strictly positive over $[0,+\infty)$, and   with $\overline{H}_1(0)=\overline{H}_2(0)=1$, and with the same support, i.e.,
  $$
  \sup\{x:\; \overline{H}_1(x)>0\}=\sup\{x:\; \overline{H}_1(x)>0\}.
  $$
Let $\overline{G}(x)$ be  the survival function defined by
\begin{equation}\label{eq:lemma:G=H1-H2}
\overline{G}(x):=\overline{H}_1\big(\overline{H}_2^{-1}(e^{-x})\big).
\end{equation}
Then
\begin{enumerate}
\item[(i)] If  $\overline{G}$ is \emph{IFRA} and $\overline{H}_1$ is \emph{DFRA}
then   $\overline{H}_2$ is \emph{DFRA}  %
\item[(ii)] If $\overline{G}$ is \emph{DFRA} and
$\overline{H}_1$ is \emph{IFRA}
then  $\overline{H}_2$ is \emph{IFRA}
\item[(iii)] If  $\overline{H}_2$ is \emph{DFRA}
and $\overline{G}$ is \emph{DFRA}    then $\overline{H}_1$ is \emph{DFRA}
\item[(iv)] If  $\overline{H}_2$ is \emph{IFRA} 
and $\overline{G}$ is \emph{IFRA}
then $\overline{H}_1$ is \emph{IFRA}
\item[(v)]  If $\overline{H}_1$ is \emph{DFRA}
and
$\overline{H}_2$ is \emph{IFRA}  %
 then
$\overline{G}$ is \emph{DFRA}
\item[(vi)] If $\overline{H}_1$ is \emph{IFRA}
and $\overline{H}_2$ is \emph{DFRA}
then
$\overline{G}$ is \emph{IFRA}
\end{enumerate}
\end{lemma}
\begin{proof}
In order to simplify the proof we deal only with the case when $\overline{H}_1$ and~$\overline{H}_2$ are strictly positive over $[0,+\infty)$.\\
As recalled in Definition~\ref{def:IFRA-DFRA}, the IFRA/DFRA property of a survival function~$\overline{H}$ is a property of  its  cumulative risk function, i.e., of the function
$$
R_{\overline{H}}(x)= - \log \big(\overline{H}(x)\big),
$$
namely that the function $R_{\overline{H}}(x)/x$ is increasing/decreasing.
With the above notation, and observing  that its inverse function is
$$
R^{-1}_{\overline{H}}(x)=\overline{H}^{-1}(e^{-x}),
$$
we can rewrite equality~(\ref{eq:lemma:G=H1-H2}) as
$$
R_{\overline{G}}(x):=R_{\overline{H}_1}\big(R_{\overline{H}_2}^{-1}(x)\big),\quad \text{or equivalently,} \quad R_{\overline{H}_1}(x)=R_{\overline{G}} \big(R_{\overline{H}_2}(x)\big).\
$$
Furthermore, we observe that
$$
\frac{R_{\overline{H}_1}(x)}{x} =\frac{R_{\overline{G}}\big(R_{\overline{H}_2}(x)\big)}{R_{\overline{H}_2}(x)}\, \frac{R_{\overline{H}_2}(x)}{x},
$$
and that the function
$$
\frac{R_{\overline{G}}\big(R_{\overline{H}_2}(x)\big)}{R_{\overline{H}_2}(x)}
$$
is increasing/decreasing if and only if $R_{\overline{G}}(x)/x$ is increasing/decreasing, i.e., if and only if $\overline{G}$ is IFRA/DFRA.
\\
Then the implications $(i)$---$(vi)$ follows immediately. For instance, when $G$ is IFRA,  $\overline{H}_1$ can be DFRA only when   $\overline{H}_2$ is DFRA.
%
\end{proof}
\begin{remark}
It is important to stress that  the simplified assumption that the two survival functions are strictly positive on $[0,+\infty)$ is not necessary: what is really important is that the two survival function are strictly positive and invertible on the the same set. Such  a property does indeed  hold when  $\overline{H}_1=\phi^{-1}$ and $\overline{H}_2=\varphi^{-1}$, with $\phi$ and $\varphi$ defined as in~(\ref{eq:phi-t-con-K-C-hat}) and (\ref{eq:varphi-t-con-K-C-hat}).  Moreover this property guarantees the desired condition that the function $\overline{G}(x)$ is strictly positive on $[0,+\infty)$.
\end{remark}

\section{Conclusions}\label{sec:conclusions}

For vectors of exchangeable non-negative random variables  with the meaning of
\textit{lifetimes}, we have considered properties of stochastic dependence
that are described in terms of the associated bivariate survival copula
$\widehat{C}$. Concerning the common, one-dimensional, marginal distribution
for the single variables, we have considered the so-called properties of
ageing, such as IFR/DFR, NBU/NWU, IFRA/DFRA,  which emerge in
different fields of applied probability.

    As it is well known, for a couple of random variables, any property of dependence as
above is generally compatible with any arbitrary choice of a one-dimensional
probability distribution. This is guaranteed by the Sklar Theorem (see, e.g.,~\cite{Nelsen2006}, ~\cite{DurSemBook2016}).
However, compatibility may fail if we assume   some extra condition on the joint probability
distribution.

For exchangeable lifetimes, some papers related with such a theme have been
published in the past years. From a modeling and subjective-probability viewpoint, the interest
toward this field had been initially motivated by the effort to extend the
property of \emph{lack of memory} of exponential distributions to the case of
a vector of dependent variables (see in particular~\cite{BarSpiz93}). A further, related, issue was the circumstance
that positive ageing properties for a family of conditional distributions can
go lost under the operation of unconditioning, as recalled in Section~\ref{section:IFR-DFR}
above.

From an  analytical  stand-point, a
suggestion was found in the results
showing  equivalence relations between dependence properties of an Archimedean copula and
ageing properties of its generator (see in particular~\cite{AveDort04}, \cite{MulScar05}).
Taking these results into account, a study of relations between
the form of marginal distributions and dependence properties of survival
copulas was developed in~\cite{BasSpiz05a}  for exchangeable survival models. Along
such a study the equivalence relations of~\cite{AveDort04}  have been extended to
obtain different types of implications between dependence and ageing, by also
introducing appropriate concepts of bivariate ageing.

Such a study had started from the observation (see also~\cite{BarSpiz93}) that   some
properties related to ageing, for a vector of exchangeable lifetimes $(X_{1},...,X_{n})$, can be translated into properties of the family of
level curves of the joint survival function of $(X_{1},...,X_{n})$. As a main tool for the description of the family of level curves,
   the ageing function $B$  was introduced. In general,
$B$ turns out to be a semi-copula. In~\cite{BasSpiz05a} it  was observed that also the
common marginal survival function $\overline{G}$ of $X_{1},...,X_{n}$ can be described
by the (Archimedean) semi-copula $S_{\overline{G}^{-1}}$.   It was observed
furthermore that univariate and bivariate properties of ageing can be
respectively characterized in terms of properties of  $\overline{G}$ and $S_{\overline
{G}^{-1}}$, provided notions of dependence are extended from copulas to
semi-copulas. This approach allowed the relations among univariate ageing,
bivariate ageing, and stochastic dependence for $( X_{1},...,X_{n})$ to be seen as relations among the dependence properties of the
three semi-copulas  $B$, $S_{\overline{G}^{-1}}$, and $\widehat{C}$ (actually
$\widehat{C}$ is always a copula). In the study of such relations,
the original equivalence relations
presented in~\cite{AveDort04}  have an inspiring role, as mentioned above.
On the other hand, in a sense, the  relations among semi-copulas can be seen as an
extension  since they convert equivalence relations into implications
by allowing properties of $B$  to enter into the game.
We point out that the mentioned results in~\cite{AveDort04} can be re-obtained by imposing the condition $B(u,v)=uv$ (i.e., Schur-constant models).

As a main motivation  of the present work, a review of this approach has been given. Along such a review we made an effort to demonstrate the logic underlying such an approach. On this purpose we  have singled out the different conceptual steps leading to  results of the type obtained in~\cite{BasSpiz05a}. We have also determined a class of  dependence properties  (\ref{C-POD-Pi})  and  corresponding ageing properties  (\ref{G-B-POD-Pi}) for which this kind of results might also hold true.
A further motivation can be found in the need to enrich the analysis of the relations between the ageing notions of  IFRA/DFRA the
dependence concepts of  PKD/NKD. On this purpose, we introduced the definition of generalized Kendall distribution   and of Kendall equivalence classes for semi-copulas. In this respect we have  pointed out the role of survival models, more general that Archimedean ones, that we called
\textquotedblleft Pseudo-Archimedean\textquotedblright.
On this basis,  the analysis of   relations between   IFRA/DFRA   and  PKD/NKD properties started in~\cite{BasSpiz05a} for the Archimedean case has been extended to the   Pseudo-Archimedean  case.  A relevant tool for such a result is  the characterization (see~\cite{GenRiv01}, and \cite{NelAl09}) of the class of  copulas that share the Kendall distribution with an Archimedean copula. We extended such a characterization to those semi-copulas  which are bivariate ageing functions of Pseudo-Archimedean models.

In this respect we also recall attention to the result, proved in~\cite{NelAl09}, showing that each Kendall equivalence class of copulas is characterized by a unique  associative copula.
Such a copula turns out to be Archimedean under the condition that we have recalled in the proof of Proposition~\ref{prop:G=phi-varphi-1}.
The result in~\cite{NelAl09} suggests a way out to  the  problem of further extending our own result from pseudo Archimedean models to general exchangeable ones.
\\
\indent  Different other types of developments and still open problems may be also suggested by
the arguments discussed in the previous sections.

One can first cite the problem of proving the extension of the analysis to multivariate notions of
dependence and of ageing as considered in~\cite{DurFosSpiz10}.

Also worth of further analysis may be the possible connections with problems
in Risk Theory. In the papers~\cite{Spreeuw10}   and~\cite{FosSpiz12} the
risk-related properties of a single-attribute utility function have been
respectively related to the dependence properties of an Archimedean copula and
to the one-dimensional ageing property of a survival function. Some conclusion
of potential interest may arise from application to these topics of the above
arguments. In~\cite{CerSpiz13}, on the other hand, the extension to
semi-copulas of dependence properties revealed also helpful in the analysis of
the \textit{mean--variance model}. This topic  recalls the attention on the
interest  of considering our approach to non-exchangeable models.

 A related problem concerns with a potential
 description of the   level sets family for
non-exchangeable, multivariate  models by means of semi-copulas. In this direction, a potentially
fruitful method can be based on replacing the function~$B$  --- which is defined
in terms of the marginal survival function $\overline{G}$ --- with a similar
function defined in terms of the trace on the diagonal.
In this respect, see also~\cite{DurMesSem06}, \cite{DurEtAl07}.
\bigskip

\noindent\textbf{Acknowledgements}
The  Authors acknowledge partial support of Ateneo Sapienza Research Projects
 \emph{Dipendenza, disuguaglianze e approssimazioni in modelli stocastici  (2015)},
 \emph{Processi stocastici: Teoria e applicazioni  (2016)},
 \emph{Simmetrie e Disuguaglianze in Modelli Stocastici  (2018)}.


\begin{thebibliography}{46}
\providecommand{\natexlab}[1]{#1}
\providecommand{\url}[1]{{#1}}
\providecommand{\urlprefix}{URL }
\expandafter\ifx\csname urlstyle\endcsname\relax
  \providecommand{\doi}[1]{DOI~\discretionary{}{}{}#1}\else
  \providecommand{\doi}{DOI~\discretionary{}{}{}\begingroup
  \urlstyle{rm}\Url}\fi
\providecommand{\eprint}[2][]{\url{#2}}

\bibitem[{1}]{Arjas1981}
Arjas, E. (1981) A stochastic process approach to multivariate reliability
  systems: notions based on conditional stochastic order. Math. Oper. Res.
  6(2):263--276, \doi{10.1287/moor.6.2.263}

\bibitem[{2}]{Arjas-Norros1984}
Arjas, E., Norros, I. (1984) Life lengths and association: a dynamic approach. Math.
  Oper. Res. 9(1):151--158, \doi{10.1287/moor.9.1.151}

\bibitem[3]{Arjas-Norros1991}
Arjas, E., Norros, I. (1991) Stochastic order and martingale dynamics in
  multivariate life length models: a review. In: Stochastic orders and decision
  under risk ({H}amburg, 1989), IMS Lecture Notes Monogr. Ser., vol~19, Inst.
  Math. Statist., Hayward, CA, pp. 7--24, \doi{10.1214/lnms/1215459846}

\bibitem[4]{AveDort04}
Av\'{e}rous, J., Dortet-Bernadet, J.L.\@ (2004) Dependence for {A}rchimedean copulas
  and aging properties of their generating functions. Sankhy\={a}
  66(4):607--620

\bibitem[5]{barlow1985bayes}
Barlow, R.E. (1985) A {B}ayes explanation of an apparent failure rate paradox.
  IEEE Transactions on Reliability 34(2):107--108

\bibitem[6]{BarMen92}
Barlow, R.E., Mendel, M.B. (1992) De {F}inetti-type {R}epresentations for {L}ife
  {D}istributions. J. Am. Stat. Ass. 87(420):1116--1122,
  \doi{10.1080/01621459.1992.10476267}

\bibitem[7]{BarPros1975}
Barlow, R.E., Proschan, F. (1975) Statistical theory of reliability and life
  testing: probability models. International series in decision processes,
  Holt, Rinehart and Winston

\bibitem[8]{BarSpiz93}
Barlow, R.E., Spizzichino, F. (1993) Schur-concave survival functions and survival
  analysis. J. Comput. Appl. Math. 46(3):437--447,
  \doi{10.1016/0377-0427(93)90039-E}

\bibitem[9]{BasSpiz99}
Bassan, B., Spizzichino, F. (1999) Stochastic comparisons for residual lifetimes
  and {B}ayesian notions of multivariate ageing. Adv. in Appl. Probab.
  31(4):1078--1094, \doi{10.1239/aap/1029955261}

\bibitem[10]{BasSpiz01}
Bassan, B., Spizzichino, F. (2001) Dependence and multivariate aging: the role of
  level sets of the survival function. In: System and {B}ayesian reliability,
  Ser. Qual. Reliab. Eng. Stat., vol~5, World Sci. Publ., River Edge, NJ, pp.
  229--242, \doi{10.1142/9789812799548\_0013}

\bibitem[11]{BasSpiz03}
Bassan, B., Spizzichino, F. (2003) On some properties of dependence and aging for
  residual lifetimes in the exchangeable case. In: Mathematical and statistical
  methods in reliability ({T}rondheim, 2002), Ser. Qual. Reliab. Eng. Stat.,
  vol~7, World Sci. Publ., River Edge, NJ, pp. 235--249,
  \doi{10.1142/9789812795250\_0016}

\bibitem[12]{BasSpiz05a}
Bassan, B., Spizzichino, F. (2005) Relations among univariate aging, bivariate
  aging and dependence for exchangeable lifetimes. J. Multivariate Anal.
  93(2):313--339, \doi{10.1016/j.jmva.2004.04.002}

\bibitem[13]{BirEsaMar1966}
Birnbaum, Z.W., Esary, J.D., Marshall, A.W. (1966) A stochastic characterization of
  wear-out for components and systems. Ann. Math. Statist. 37:816--825,
  \doi{10.1214/aoms/1177699362}

\bibitem[14]{CerSpiz13}
Cerqueti, R., Spizzichino, F. (2013) Extension of dependence properties to
  semi-copulas and applications to the mean-variance model. Fuzzy Sets and
  Systems 220:99--108, \doi{10.1016/j.fss.2012.08.011}

\bibitem[15]{DurGhis09}
Durante, F., Ghiselli Ricci, R.  (2009) Supermigrative semi-copulas and triangular norms.
  Inform. Sci. 179(15):2689--2694, \doi{10.1016/j.ins.2009.04.001}


\bibitem[16]{DurSem05-K}
Durante, F., Sempi, C. (2005) Semicopul\ae. Kybernetika (Prague) 41(3):315--328

\bibitem[17]{DurSem05-I}
Durante, F., Sempi, C. (2005) Copula and semicopula transforms. Int. J. Math. Math. Sci.
  (4):645--655, \doi{10.1155/IJMMS.2005.645}

\bibitem[18]{DurSemBook2016}
Durante, F., Sempi, C. (2016) Principles of copula theory. CRC Press, Boca Raton,
  FL

\bibitem[19]{DurSpiz10}
Durante, F., Spizzichino, F. (2010) Semi-copulas, capacities and families of level
  sets. Fuzzy Sets and Systems 161(2):269--276, \doi{10.1016/j.fss.2009.03.002}

\bibitem[20]{DurMesSem06}
Durante, F., Mesiar, R., Sempi, C. (2006) On a family of copulas constructed from the
  diagonal section. Soft. Comput. 10(6):490--494, \doi{10.1007/s00500-005-0523-7}

\bibitem[21]{DurEtAl07}
Durante, F., Koles\'{a}rov\'{a}, A., Mesiar, R., Sempi C. (2007) Copulas with given
  diagonal sections: novel constructions and applications. Internat. J. Uncertain
  Fuzziness Knowledge-Based Systems 15(4):397--410,
  \doi{10.1142/S0218488507004753}

\bibitem[22]{DurFosSpiz10}
Durante, F., Foschi, R., Spizzichino, F. (2010) Aging functions and multivariate
  notions of {NBU} and {IFR}. Probab. Engrg. Inform. Sci. 24(2):263--278,
  \doi{10.1017/S026996480999026X}

\bibitem[23]{Def1937}
de~Finetti, B. (1937) La pr\'{e}vision: ses lois logiques, ses sources
  subjectives. Ann. Inst. H. Poincar\'{e} 7(1):1--68

\bibitem[24]{FosSpiz12}
Foschi, R., Spizzichino, F. (2012) Interactions between ageing and risk properties
  in the analysis of burn-in problems. Decis. Anal. 9(2):103--118,
  \doi{10.1287/deca.1120.0236}

\bibitem[25]{GenRiv1993}
Genest, C., Rivest, L.P. (1993) Statistical inference procedures for bivariate
  {A}rchimedean copulas. J. Amer. Statist. Assoc. 88(423):1034--1043

\bibitem[26]{GenRiv01}
Genest, C., Rivest, L.P. (2001) On the multivariate probability integral
  transformation. Statist. Probab. Lett. 53(4):391--399

\bibitem[27]{Jan-Konst16repres}
Janson, S., Konstantopoulos, T., Yuan, L. (2016) On a representation theorem for
  finitely exchangeable random vectors. Journal of Mathematical Analysis and
  Applications 442(2):703--714

\bibitem[28]{Joe1997}
Joe, H. (1997) Multivariate models and dependence concepts, Monographs on
  Statistics and Applied Probability, vol~73. Chapman \& Hall, London,
  \doi{10.1201/b13150}

\bibitem[29]{kerns-et-al2006}
Kerns, J.G., Sz{\'e}kely, G.J. (2006) De {F}inetti's {T}heorem for {A}bstract
  {F}inite {E}xchangeable {S}equences. Journal of Theoretical Probability
  19(3):589--608

\bibitem[30]{KimSam1989}
Kimeldorf, G., Sampson, A.R. (1989) A framework for positive dependence. Ann. Inst.
  Statist. Math. 41(1):31--45,
  \\\url{https://link.springer.com/content/pdf/10.1007\%2FBF00049108.pdf}

\bibitem[31]{LaiXie2006}
Lai, C.D., Xie, M. (2006) Stochastic ageing and dependence for reliability.
  Springer, New York, with a foreword by Richard E. Barlow

\bibitem[32]{Leon16}
Leonetti, P. (2016) Finite {P}artially {E}xchangeable {L}aws are {S}igned
  {M}ixtures of {P}roduct {L}aws. Sankhya A, pp. 1--20

\bibitem[33]{LiLi2013}
Li, H., Li, X. (2013) Stochastic {O}rders in {R}eliability and {R}isk {M}anagement.
  In {H}onor of {P}rofessor {M}oshe {S}haked. Lecture Notes in Statistics,
  Springer, New York

\bibitem[34]{MarshOlk1979}
Marshall, A.W., Olkin, I. (1979) Inequalities: {T}heory of {M}ajorization and its
  {A}pplications, Mathematics in Science and Engineering, vol. 143. Academic
  Press, Inc. [Harcourt Brace Jovanovich, Publishers], New York-London

\bibitem[35]{MulScar05}
M\"{u}ller, A., Scarsini, M. (2005) Archimedean copulae and positive dependence. J.
  Multivariate Anal. 93(2):434--445, \doi{10.1016/j.jmva.2004.04.003}

\bibitem[36]{NapSpiz09}
Nappo, G., Spizzichino, F. (2009) Kendall dis\-tri\-b\-u\-tions and level sets in bi\-variate
  ex\-change\-able sur\-vival models. Inform. Sci. 179(17):2878--2890,
  \doi{10.1016/j.ins.2009.02.007}

\bibitem[37]{Nelsen2006}
Nelsen, R.B. (2006) An introduction to copulas, 2nd edn. Springer Series in
  Statistics, Springer, New York

\bibitem[38]{NelAl03}
Nelsen, R.B., Quesada-Molina, J.J., Rodr{\'{\i}}guez-Lallena, J.A., {\'U}beda-Flores, M.
  (2003) Kendall distribution functions. Statist. Probab. Lett. 65(3):263--268

\bibitem[39]{NelAl09}
Nelsen, R.B., Quesada-Molina, J.J., Rodr\'{\i}guez-Lallena, J.A., \'{U}beda Flores, M.
  (2009) Kendall distribution functions and associative copulas. Fuzzy Sets and
  Systems 160(1):52--57, \doi{10.1016/j.fss.2008.05.001}

\bibitem[40]{ScarSpiz99}
Scarsini, M., Spizzichino, F. (1999) Simpson-type paradoxes, dependence, and
  ageing. J. Appl. Probab. 36(1):119--131, \doi{10.1239/jap/1032374234}

\bibitem[41]{ShaSha2007}
Shaked, M., Shanthikumar, J.G. (2007) Stochastic orders. Springer Series in
  Statistics, Springer, New York, \doi{10.1007/978-0-387-34675-5}

\bibitem[42]{Spiz92}
Spizzichino, F. (1992) Reliability decision problems under conditions of ageing.
  In: Bayesian statistics, 4 ({P}e\~{n}\'{i}scola, 1991), Oxford Univ. Press,
  New York, pp. 803--811

\bibitem[43]{Spiz01}
Spizzichino, F. (2001) Subjective probability models for lifetimes, Monographs on
  Statistics and Applied Probability, vol~91. Chapman \& Hall/CRC, Boca Raton,
  FL, \doi{10.1201/9781420036138}

\bibitem[44]{Spiz10}
Spizzichino, F. (2010) Semi-copulas and interpretations of coincidences between
  stochastic dependence and ageing. In: Copula theory and its applications,
  Lect. Notes Stat. Proc., vol. 198, Springer, Heidelberg, pp. 237--254,
  \doi{10.1007/978-3-642-12465-5\_11}

\bibitem[45]{Spiz14}
Spizzichino, F. (2014) Aging and {P}ositive {D}ependence, \emph{in Wiley
  StatsRef: Statistics Reference Online}, American Cancer Society.
  \doi{10.1002/9781118445112.stat03944}

\bibitem[46]{Spreeuw10}
Spreeuw, J. (2010) Relationships between {A}rchimedean copulas and {M}orgenstern
  utility functions. In: Copula theory and its applications, Lect. Notes Stat.
  Proc., vol. 198, Springer, Heidelberg, pp. 311--322,
  \doi{10.1007/978-3-642-12465-5\_17}

\end{thebibliography}
\end{document}